\documentclass[twoside,a4paper,english, 10pt]{amsart}

\input{artmymacros.sty}
	
\begin{document}

\title[$G$-valued crystalline representations]{$G$-valued crystalline representations with minuscule $p$-adic Hodge type}
    \author{Brandon Levin}

\begin{abstract}
We study $G$-valued semi-stable Galois deformation rings where $G$ is a reductive group. We develop a theory of Kisin modules with $G$-structure and use this to identify the connected components of crystalline deformation rings of minuscule $p$-adic Hodge type with the connected components of moduli of ``finite flat models with $G$-structure.'' The main ingredients are a construction in integral $p$-adic Hodge theory using Liu's theory of $(\phz, \widehat{G})$-modules and the local models constructed by Pappas and Zhu.  
\end{abstract}

\maketitle

\tableofcontents

\section{Introduction}

\newtheorem*{thm:locmodels}{Theorem \ref{locmodels}}
\newtheorem*{thm:main}{Theorem \ref{main}}
\newtheorem*{cor:unram}{Theorem \ref{connectedunramified}}

\subsection{Overview}

One of the principal challenges in the study of modularity lifting or more generally automorphy lifting via the techniques introduced in Taylor-Wiles \cite{TW} is understanding local deformation conditions at $\ell = p$. In \cite{MFFGS}, Kisin introduced a ground-breaking new technique for studying one such condition, flat deformations, which led to better modularity lifting theorems.  \cite{PST} extends those techniques to construct potentially semistable deformation rings with specified Hodge-Tate weights. In this paper, we study Galois deformations valued in a reductive group $G$ and extend Kisin's techniques to this setting. In particular, we define and prove structural results about ``flat'' $G$-valued deformations.

Let $G$ be a reductive group over a $\Zp$-finite flat local domain $\La$ with connected fibers. Let $\F$ be the residue field of $\La$ and $F := \La[1/p]$. Let $K/\Qp$ be a finite extension with absolute Galois group $\Gamma_K$ and fix a representation $\overline{\eta}:\Gamma_K \ra G(\F)$. The (framed) $G$-valued deformation functor is represented by a complete local Noetherian $\La$-algebra $R^{\square}_{G, \overline{\eta}}$.  For any geometric cocharacter $\mu$ of $\Res_{(K \otimes_{\Qp} F)/F} G_F$, there exists a quotient
$R^{\st, \mu}_{\overline{\eta}}$ (resp. $R^{\cris, \mu}_{\overline{\eta}}$) of  $R^{\square}_{G, \overline{\eta}}$ whose points over finite extensions $F'/F$ are semi-stable (resp. crystalline) representations with $p$-adic Hodge type $\mu$ (see \cite[Theorem 4.0.12]{Balaji}).    

When $G = \GL_n$ and $\mu$ is minuscule, $R^{\cris, \mu}_{\overline{\eta}}$ is a quotient of a flat deformation ring. For modularity lifting, it is important to know the connected components of $\Spec R^{\cris, \mu}_{\overline{\eta}}[1/p]$.  Intuitively, Kisin's technique introduced in \cite{MFFGS} is to resolve the flat deformation ring by ``moduli of finite flat models'' of deformations of $\overline{\eta}$. When $K/\Qp$ is ramified, the resolution is not smooth, but its singularities are relatively mild, which allowed for the determination of the connected components  in many instances when $G = \GL_2$ \cite[2.5.6, 2.5.15]{MFFGS}.  Kisin's technique extends beyond the flat setting (for $\mu$ arbitrary) where one resolves deformation rings by moduli spaces of integral $p$-adic Hodge theory data called $\fS$-modules of finite height also known as \emph{Kisin modules}. 

In this paper, we define a notion of Kisin module with $G$-structure or as we call them \emph{$G$-Kisin modules} (Definition \ref{GKisinmodule}), and we construct a resolution
$$
\Theta:X^{\cris, \mu}_{\ebar} \ra \Spec R^{\cris, \mu}_{\overline{\eta}}    
$$ 
where $\Theta$ is a projective morphism and $\Theta[1/p]$ is an isomorphism (see Propositions \ref{resolution}, \ref{fhimage}) .  The same construction works for $R^{\st, \mu}_{\overline{\eta}}$ as well. The goal then is to understand the singularities of $X^{\cris, \mu}_{\ebar}$.  The natural generalization of the flat condition for $\GL_n$ to an arbitrary group $G$ is \emph{minuscule} $p$-adic Hodge type $\mu$.  A cocharacter $\mu$ of a reductive group $H$ is minuscule if its weights when acting on $\Lie H$ lie in $\{-1, 0, 1\}$ (see Definition \ref{liegrading} and discussion afterward). Our main theorem is a generalization of the main result of \cite{MFFGS} on the geometry of $X^{\cris, \mu}_{\ebar}$ for $G$ reductive and $\mu$ minuscule:    

\begin{thm:main} Assume $p \nmid \pi_1(G^{\der})$ where $G^{\der}$ is the derived subgroup of $G$. Let $\mu$ be a minuscule geometric cocharacter of $\Res_{(K \otimes_{\Qp} F)/F} G_F$. Then $X^{\cris, \mu}_{\overline{\eta}}$ is normal and $X^{\cris, \mu}_{\overline{\eta}} \otimes_{\La_{[\mu]}} \F_{[\mu]}$ is reduced where $\La_{[\mu]}$ is the ring of integers of the reflex field of $\mu$. 
\end{thm:main} 

When $G = \mathrm{GSp}_{2g}$, this is a result of Broshi \cite{Broshi}; also, this is a stronger result than in \cite{LevinThesis} where we made a more restrictive hypothesis on $\mu$ (see Remark \ref{assume}).  The significance of Theorem \ref{main} is that it allows one to identify the connected components of $\Spec R^{\cris, \mu}_{\overline{\eta}}[1/p]$ with the connected components of the fiber in $X^{\cris, \mu}_{\overline{\eta}}$ over the closed point of $\Spec R^{\cris, \mu}_{\overline{\eta}}$, a projective scheme over $\F_{[\mu]}$ (see Corollary \ref{connectedcomponents}). This identification led to the successful determination of the connected components of $\Spec R^{\cris, \mu}_{\overline{\eta}}[1/p]$ in the case when $G = \GL_2$ (\cite{MFFGS, GeeFlat, Imai1, Imai2, HellmannFlat}).  Outside of $\GL_2$, relatively little is known about the connected components of these deformations rings without restricting the ramification in $K$.

 When $K/\Qp$ is unramified, we have a stronger result: 
\begin{cor:unram} Assume $K/\Qp$ is unramified, $p > 3$, and $p \nmid \pi_1(G^{\mathrm{ad}})$.  Then the universal crystalline deformation ring $R^{\cris, \mu}_{\overline{\eta}}$ is formally smooth over $\La_{[\mu]}$. In particular, $\Spec R^{\cris, \mu}_{\overline{\eta}}[1/p]$ is connected. 
\end{cor:unram}   

\begin{rmk} \label{assume} In \cite{LevinThesis}, we made the assumption on the cocharacter $\mu$ that there exists a representation $\rho:G \ra \GL(V)$ such that  $\rho \circ \mu$ is minuscule. This extra hypothesis on $\mu$ excluded most adjoint groups like $\PGL_n$ as well as exceptional types like $E_6$ and $E_7$ both of which have minuscule cocharacters.  One can weaken the assumptions in Theorem \ref{connectedunramified} if one assumes this stronger condition on $\mu$.
\end{rmk}

\begin{rmk} The groups $\pi_1(G^{\der})$ and $\pi_1(G^{\mathrm{ad}})$ appearing in Theorems \ref{main} and \ref{connectedunramified} are the fundamental groups in the sense of semisimple groups.  Note that $\pi_1(G^{\der})$ is a subgroup of $\pi_1(G^{\mathrm{ad}})$.  The assumption that $p \nmid \pi_1(G^{\der})$ insures that the local models we use have nice geometric property.  The stronger assumption in Theorem \ref{connectedunramified} that $p \nmid \pi_1(G^{\mathrm{ad}})$ is probably not necessary and is a byproduct of the argument which involves reduction to the adjoint group.
\end{rmk}   

There are two main ingredients in the proof of Theorem \ref{main} and its applications, one coming from integral $p$-adic Hodge theory and the other from local models of Shimura varieties. In Kisin's original construction, a key input was an advance in integral $p$-adic Hodge theory, building on work of Breuil, which allows one to describe finite flat group schemes over $\cO_K$ in terms of certain linear algebra objects called \emph{Kisin modules} of height in $[0,1]$ (\cite{MFFGS, Fcrystals}). More precisely, then, $X^{\cris, \mu}_{\overline{\eta}}$ is a moduli space of $G$-Kisin modules with ``type'' $\mu$.  Intuitively, one can imagine $X^{\cris, \mu}_{\overline{\eta}}$ as a moduli of finite flat models with additional structure.  

The proof of Theorem \ref{main} uses a recent advance of Liu \cite{Liu1} in integral $p$-adic Hodge theory to overcome a difficulty in identifying the local structure of $X^{\cris, \mu}_{\overline{\eta}}$. Heuristically, the difficulty arises because for a general group $G$ one cannot work only in the setting of Kisin modules of height in $[0,1]$ where one has a nice equivalence of categories between finite flat group schemes and the category of Kisin modules with height in $[0,1]$. Beyond the height in $[0,1]$ situation, the Kisin module only remembers the Galois action of the subgroup $\Gamma_{\infty} \subset \Gamma_K$ which fixes the field $K( \pi^{1/p}, \pi^{1/p^2}, \ldots)$ for some compatible system of $p$-power roots of  uniformizer $\pi$ of $K$. 

Liu \cite{Liu1} introduced a more complicated linear algebra structure on a Kisin module, called a $(\phz, \widehat{G})$-module, which captures the action of the full Galois group $\Gamma_K$. We call them $(\phz, \widehat{\Gamma})$-modules to avoid confusion with the group $G$. Let $A$ be a finite local $\La$-algebra which is either Artinian or flat. Our principal result (Theorem \ref{keythm}) says roughly that if $\rho:\Gamma_{\infty} \ra G(A)$ has ``type'' $\mu$, i.e., comes from a $G$-Kisin module $(\fP_A, \phi_A)$ over $A$ of type $\mu$ with $\mu$ minuscule, then there exists a canonical extension $\widetilde{\rho}:\Gamma_K \ra G(A)$ and furthermore if $A$ is flat over $\Zp$ then $\widetilde{\rho}[1/p]$ is crystalline. This is rough in the sense that what we actually prove is an isomorphism of certain deformation functors. As a consequence, we get that the local structure of $X^{\cris, \mu}_{\overline{\eta}}$ at a point $(\fP_{\F'}, \phi_{\F'}) \in X^{\cris, \mu}_{\overline{\eta}}(\F')$ is smoothly equivalent to the deformation groupoid $D^{\mu}_{\fP_{\F'}}$ of $\fP_{\F'}$ with type $\mu$.   

To prove Theorem \ref{main}, one studies the geometry of $D^{\mu}_{\fP_{\F'}}$. Here, the key input comes from the theory of local models of Shimura varieties. A \emph{local model} is a projective scheme $X$ over the ring of integers of a $p$-adic field $F$ such that $X$ is supposed to \'etale-locally model the integral structure of a Shimura variety. Classically, local models were built out of moduli spaces of linear algebra structures. Rapoport and Zink \cite{RZ} formalized the theory of local models for Shimura varieties of PEL-type. Subsequent refinements of these local models were studied mostly on a case by case basis by Faltings, G\"ortz, Haines, Pappas, and Rapoport, among others.  

Pappas and Zhu \cite{PZ} define for any triple $(G, P, \mu)$, where $G$ is a reductive group over $F$ (which splits over a tame extension), $P$ is a parahoric subgroup, and $\mu$ is any cocharacter of $G$, a local model $M(\mu)$ over the ring of integers of the reflex field of $\mu$. Their construction, unlike previous constructions, is purely group-theoretic, i.e. it does not rely on any particular representation of $G$. They build their local models inside degenerations of affine Grassmanians extending constructions of Beilinson, Drinfeld, Gaitsgory, and Zhu to mixed characteristic.  The geometric fact we will use is that $M(\mu)$ is normal with special fiber reduced (\cite[Theorem 0.1]{PZ}).

The significance of local models in this paper is that the singularities of $X^{\cris, \mu}_{\overline{\eta}}$ are smoothly equivalent to those of a local model $M(\mu)$ for the Weil-restricted group $\Res_{(K \otimes_{\Qp} F)/F} G_F$. This equivalence comes from a diagram of formally smooth morphisms (\ref{musmoothmod}):
\begin{equation} \label{b7}
\xymatrix{
& \widetilde{D}^{(\infty), \mu}_{\fP_{\F}} \ar[dl] \ar[dr] & \\
D^{\mu}_{\fP_{\F}} & & \overline{D}^{\mu}_{Q_{\F}}, \\
}
\end{equation}
which generalizes constructions from \cite[2.2.11]{MFFGS} and \cite[\S 3]{PRcoeff}. The deformation functor $\overline{D}^{\mu}_{Q_{\F}}$ is represented by the completed local ring at an $\F$-point of $M(\mu)$. Intuitively, the above modification corresponds to adding a trivialization to the $G$-Kisin module and then taking the ``image of Frobenius.'' We construct the diagram (\ref{b7}) in \S 3 with no assumptions on the cocharacter $\mu$ (to be precise $D^{\mu}_{\fP_{\F}}$ is deformations of type $\leq \mu$ in general). It is intriguing to wonder whether $D^{\mu}_{\fP_{\F}}$ and diagram (\ref{b7}) has any relevance to studying higher weight Galois deformation rings, i.e., when $\mu$ is not minuscule.    

As a remark, we usually cannot apply \cite{PZ} directly since the group $\Res_{(K \otimes_{\Qp} F)/F} G$ will generally not split over a tame extension.  In \cite{LevinThesis}, we develop a theory of local models following Pappas and Zhu's approach but adapted to these Weil-restricted groups (for maximal special parahoric level).  These results are reviewed in \S 3.2 and are studied in more generality in \cite{LevinLM}. 

We now a give brief outline of the article.  In \S 2, we define and develop the theory of $G$-Kisin modules and construct resolutions of semi-stable and crystalline $G$-valued deformation rings (\ref{resolution}, \ref{fhimage}).  This closely follows the approach of \cite{PST}.  The proof that ``semi-stable implies finite height'' (Proposition \ref{Gssimpliesfh}) requires an extra argument not present in the $\GL_n$-case (Lemma \ref{extensionlemma}).  In \S 3, we study the relationship between deformations of $G$-Kisin modules and local models.  We construct the big diagram (Theorem \ref{Modification}) and then impose the $\mu$-type condition to arrive at the diagram (\ref{musmoothmod}).  We also give an initial description of the local structure of $X^{\cris, \mu}_{\overline{\eta}}$ in Corollary \ref{cristypes}.  \S 4.1 develops the theory of $(\phz, \widehat{\Gamma})$-modules with $G$-structure, and  \S 4.2 is devoted to the proof of our key result (Theorem \ref{keythm}) in integral $p$-adic Hodge theory.  In the last section \S 4.3, we prove Theorems \ref{main} and \ref{connectedunramified} which follow relatively formally from the results of \S 3.3 and \S 4.2.    

\subsection{Acknowledgments}  
The enormous influence of the work of Mark Kisin and Tong Liu in this paper will be evident to the reader.   This paper is based on the author's Stanford Ph.D. thesis advised by Brian Conrad to whom the author owes a debt of gratitude for his generous guidance and his feedback on multiple drafts of the thesis.  It is a pleasure to thank Tong Liu, Xinwen Zhu, George Pappas, Bhargav Bhatt, and Mark Kisin for many helpful discussions and exchanges related to this project.  The author is additionally grateful for the support of the National Science Foundation and the Department of Defense in the form of NSF and NDSEG fellowships.  Part of this work was completed while the author was a visitor at the Institute for Advanced Study supported by the National Science Foundation grant DMS-1128155.  The author is grateful to the IAS for its support and hospitality. Finally, we would like to thank the referee for a number of valuable suggestions. 
\subsection{Notations and conventions}
We take $F$ to be our coefficient field, a finite extension of $\Qp$.  Let $\La$ be the ring of integers of $F$ with residue field $\F$.   Let $G$ be reductive group scheme over $\La$ with connected fibers and $\sideset{^f}{_{\La}}\Rep (G)$ the category of representations of $G$ on finite free $\La$-modules. We will use $V$ to denote a fixed faithful representation of $G$, i.e., $V \in \sideset{^f}{_{\La}}\Rep (G)$ such that $G \ra \GL(V)$ is a closed immersion.  The derived subgroup of $G$ will be denoted by $G^{\der}$ and its adjoint quotient by $G^{\ad}$. 

All $G$-bundles will be with respect to the fppf topology. If $X$ is a $\La$-scheme, then $\GBun(X)$ will denote the category of $G$-bundles on $X$. We will denote  the trivial $G$-bundle by $\cE^0$. For any $G$-bundle $P$ on a $\La$-scheme $X$ and any $W \in \sideset{^f}{_{\La}}\Rep (G)$, $P(W)$ will denote the pushout of $P$ with respect to $W$ (see discussion before Theorem \ref{catbundle}).  Let $\overline{F}$ be an algebraic closure of $F$.  For a linear algebraic $F$-group $H$, $X_*(H)$ will denote the group of geometric cocharacters, i.e., $\Hom(\Gm, H_{\overline{F}})$.  For $\mu \in X_*(H)$, $[\mu]$ will denote its conjugacy class . The reflex field $F_{[\mu]}$ of $[\mu]$ is the smallest subfield of $\overline{F}$ over which the conjugacy class $[\mu]$ is defined.

If $\Gamma$ is a pro-finite group and $B$ is a finite $\La$-algebra, then $\sideset{^f}{_{B}}\Rep (\Gamma)$ will be the category of continuous representations of $\Gamma$ on finite projective $B$-modules where $B$ is given the $p$-adic topology.  More generally, $\GRep_{B}(\Gamma)$ will denote the category of pairs $(P, \eta)$ where $P$ is a $G$-bundle over $\Spec B$ and $\eta:\Gamma \ra \Aut_G(P)$ is a continuous homomorphism.

Let $K$ be a $p$-adic field with rings of integers $\cO_K$ and residue field $k$. Denote its absolute Galois group by $\Gamma_K$.  We furthermore take $W := W(k)$ and $K_0 := W[1/p]$.  We fix a uniformizer $\pi$ of $K$ and let $E(u)$ the minimal polynomial of $\pi$ over $K_0$. Our convention will be to work with covariant $p$-adic Hodge theory functors so we take the $p$-adic cyclotomic character to have Hodge-Tate weight $-1$.  

For any local ring $R$, we let $m_R$ denote the maximal ideal.  We will denote the completion of $B$ with respect to a specified topology by $\widehat{B}$.

\section{Kisin modules with $G$-structure}

In this section, we construct resolutions of Galois deformation rings by moduli spaces of Kisin modules (i.e. $\fS$-modules) with $G$-structure. For $\GL_n$, this technique was introduced in \cite{MFFGS} to study flat deformation rings. In \cite{PST}, the same technique is used to construct potentially semi-stable deformation rings for $\GL_n$. Here we develop a theory of $G$-Kisin modules (Definition \ref{GKisinmodule}). In particular, in \S 2.4, we show the existence of a universal $G$-Kisin module over these deformation rings (Theorem \ref{universalKisin}) and relate the filtration defined by a $G$-Kisin module to $p$-adic Hodge type.  One can construct $G$-valued semi-stable and crystalline deformation rings with fixed $p$-adic Hodge type without  $G$-Kisin modules \cite{Balaji}.   However, the existence of a resolution by a moduli space of Kisin modules allows for finer analysis of the deformation rings as is carried out in \S 4.        

\subsection{Background on $G$-bundles}

All bundles will be for the fppf topology. For any $G$-bundle $P$ on a $\La$-scheme $X$ and any $W \in \sideset{^f}{_{\La}}\Rep (G)$, define
$$
P(W) := P \times^G W = (P \times W)/\sim
$$
to be the pushout of $P$ with respect to $W$. This is a vector bundle on $X$.  This defines a functor from $\sideset{^f}{_{\La}}\Rep (G)$ to the category  $\Vect_X$ of vector bundles on $X$.

\begin{thm} \label{catbundle} Let $G$ be a flat affine group scheme of finite type over $\Spec \La$ with connected fibers. Let $X$ be an $\La$-scheme. The functor $P \mapsto \{P(W)\}$ from the category of $G$-bundles on $X$ to the category of fiber functors $($i.e., faithful exact tensor functors$)$ from $\sideset{^f}{_{\La}}\Rep (G)$ to $\Vect_X$ is an equivalence of categories.  
\end{thm}
\begin{proof}
When the base is a field, this is a well-known result (\cite[Theorem 3.2]{DM}) in Tannakian theory.   When the base is a Dedekind domain, see \cite[Theorem 4.8]{BroshiGtorsors} or \cite[Theorem 2.5.2]{LevinThesis}.
\end{proof}

We will also need the following gluing lemma for $G$-bundles:
\begin{lemma} \label{descent} Let $B$ be any $\La$-algebra. Let $f \in B$ be a non-zero divisor and $G$ be a flat affine group scheme of finite type over $\La$. The category of triples $(P_f, \widehat{P}, \alpha)$, where $P_f \in \GBun(\Spec B_f)$, $\widehat{P} \in \GBun(\Spec \widehat{B})$, and $\alpha$ is an isomorphism between $P_f$ and $\widehat{P}$ over $\Spec \widehat{B}_f$, is equivalent to the category of $G$-bundles on $B$. 
\end{lemma}
\begin{proof} This is a generalization of the Beauville-Laszlo formal gluing lemma for vector bundles.  See \cite[Lemma 5.1]{PZ} or \cite[Theorem 3.1.8]{LevinThesis}.
\end{proof}

Let $i:H \subset G$ be a flat closed $\La$-subgroup.  We are interested in the ``fibers'' of the pushout map
$$
i_*:\HBun \ra \GBun
$$  
carrying an $H$-bundle $Y$ to the $G$-bundle $Y \times^H G$. Let $Q$ be a $G$-bundle on a $\La$-scheme $S$. For any $S$-scheme $X$, define $\mathrm{Fib}_Q(X)$ to be the category of pairs $(P, \alpha)$, where $P \in \HBun(X)$ and $\alpha:i_*(P) \cong Q_X$ is an isomorphism in $\GBun(X)$. A morphism $(P, \alpha) \ra (P', \alpha')$ is a map $f:P \ra P'$ of $H$-bundles such that $\alpha' \circ i_*(f) \circ \alpha^{-1}$ is the identity.

\begin{prop} \label{changeofgroup} The category of $\mathrm{Fib}_Q(X)$ has no non-trivial automorphisms for any $S$-schemes $X$. Furthermore, the underlying functor $|\mathrm{Fib}_Q|$ is represented by the pushout $Q \times^G (G / H)$.
In particular, if $G / H$ is affine $($resp. quasi-affine$)$ over $S$ then $|\mathrm{Fib}_Q|$ is affine $($resp. quasi-affine$)$ over $X$.
\end{prop}
\begin{proof} See \cite[Proposition 9]{SerreBundle} or \cite[Lemma 2.2.3]{LevinThesis}. 
\end{proof}

\begin{prop} \label{Gtriv} Let $G$ be a smooth affine group scheme of finite type over $\Spec \La$ with connected fibers. 
\begin{enumerate} 
\item Let $R$ any $\La$-algebra and $I$ a nilpotent ideal of $R$.  For any $G$-bundle $P$ on $\Spec R$, $P$ is trivial if and only if $P \otimes_R R/I$ is trivial.  
\item Let $R$ be any complete local $\La$-algebra with finite residue field.   Any $G$-bundle on $\Spec R$ is trivial.  
\end{enumerate}
\end{prop}
\begin{proof} For (1), because $G$ is smooth, $P$ is also smooth.   Thus, $P(R) \ra P(R/I)$ is surjective.  A $G$-bundle is trivial if and only if it admits a section.  

Part (2) reduces to the case of $R = \F$ using part (1).   Lang's Theorem says that $H^1_{\et} (\F, G)$ is trivial for any smooth connected algebraic group over $\F$ (see Theorem 4.4.17 \cite{Springer})
\end{proof}  

\subsection{Definitions and first properties}
  
Let $K$ be a $p$-adic field with rings of integers $\cO_K$ and residue field $k$. Set $W := W(k)$ and $K_0 := W[1/p]$. Recall Breuil/Kisin's ring $\fS := W[\![u]\!]$ and let $E(u) \in W[u]$ be the Eisenstein polynomial associated to a choice of uniformizer $\pi$ of $K$ which generates $K$ over $K_0$. Fix a compatible system $\{\pi^{1/p}, \pi^{1/p^2}, \ldots \}$ of $p$-power roots of $\pi$ and let $K_{\infty} = K(\pi^{1/p}, \pi^{1/p^2}, \ldots )$.  Set $\Gamma_{\infty} := \Gal(\overline{K}/K_{\infty})$.  

Let $\cO_{\cE}$ denote the $p$-adic completion of $\mathfrak{S}[1/u]$. We equip both $\cO_{\cE}$ and $\fS$ with a Frobenius endomorphism $\varphi$ defined by taking the ordinary Frobenius lift on $W$ and $u \mapsto u^p$. For any $\Zp$-algebra $B$, let $\cO_{\cE, B} := \cO_{\cE} \otimes_{\Zp} B$ and $\fS_B := \fS \otimes_{\Zp} B$. We equip both of these rings with Frobenii having trivial action on $B$.  Note that all tensor products are over $\Zp$ even though the group $G$ may only be defined over the $\La$. 

\begin{defn} \label{GOE} Let $B$ be any $\La$-algebra. For any $G$-bundle on $\Spec \cO_{\cE, B}$, we take $\phz^*(P) := P \otimes_{\cO_{\cE, B}, \phz} \cO_{\cE, B}$ to be the pullback under Frobenius. An \emph{$(\cO_{\cE, B}, \varphi)$-module with $G$-structure} is a pair $(P, \phi_{P})$ where $P$ is a $G$-bundle on $\Spec \cO_{\cE, B}$ and $\phi_{P}:\phz^*(P) \cong P$ is an isomorphism. Denote the category of such pairs by $\GMod^{\varphi}_{\cO_{\cE, B}}$.
\end{defn}

\begin{rmk} When $G = \GL_d$, $\GMod^{\varphi}_{\cO_{\cE, B}}$ is equivalent to the category of rank $d$ \'etale $(\cO_{\cE, B}, \phz)$-modules via the usual equivalence between $\GL_d$-bundles and rank $d$ vector bundles.
\end{rmk}    

When $B$ is $\Zp$-finite and Artinian, the functor $T_B$ defined by $$T_B(M, \phi) = (M \otimes_{\cO_{\cE}} \cO_{\widehat{\cE}^{\un}})^{\phi = 1}$$ induces an equivalence of categories between \'etale $(\cO_{\cE, B}, \phz)$-modules (which are $\cO_{\cE, B}$-projective) and the category of representations of $\Gamma_{\infty}$ on  finite projective $B$-modules (see \cite[Lemma 1.2.7]{MFFGS}). A quasi-inverse is given by
$$
\underline{M}_B(V) := (V \otimes_{\Zp} \cO_{\widehat{\cE}^{\un}})^{\Gamma_{\infty}}.
$$
This equivalence extends to algebras which are finite flat over $\Zp$.

\begin{defn} \label{GREP} For any profinite group $\Gamma$ and any $\La$-algebra $B$, define $\GRep_B(\Gamma)$ to be the category of pairs $(P, \eta)$ where $P$ is a $G$-bundle over $\Spec B$ and $\eta:\Gamma \ra \Aut_G(P)$ is a continuous homomorphism (where $B$ is given the $p$-adic topology).
\end{defn} 

In the $G$-setting, $\GRep_B(\Gamma)$ will play the role of representation of $\Gamma$ on finite projective $B$-modules.  We have the following generalization of $T_B$:

\begin{prop} \label{normGequiv} Let $B$ be any $\La$-algebra which is $\Zp$-finite and either Artinian or $\Zp$-flat. There exists an equivalence of categories
$$
T_{G, B}:\GMod^{\phz}_{\cO_{\cE, B}} \ra \GRep_B(\Gamma_{\infty})
$$
with a quasi-inverse $\underline{M}_{G, B}$. Furthermore, for any finite map $B \ra B'$ and any $(P, \phi_P) \in \GMod^{\phz}_{\cO_{\cE, B}}$, there is a natural isomorphism
$$
T_{G, B'}(P \otimes_B B') \cong T_{G, B}(P) \otimes_B B'.
$$ 
\end{prop}
\begin{proof} Using Theorem \ref{catbundle}, we can give Tannakian interpretations of both $\GMod^{\phz}_{\cO_{\cE, B}}$ and $\GRep_B(\Gamma_{\infty})$. The former is equivalent to the category $[\sideset{^f}{_{\La}}\Rep (G),\Mod^{\varphi, \et}_{\cO_{\cE, B}}]^{\otimes}$ of faithful exact tensor functors. The latter is equivalent to the category of faithful exact tensor functors from $\sideset{^f}{_{\La}}\Rep (G)$ to $\sideset{^f}{_{B}}\Rep (\Gamma_{\infty})$.  We define $T_{G, B}(P, \phi_P)$ to be the functor which assigns to any $W \in \sideset{^f}{_{\La}}\Rep (G)$ the $\Gamma_{\infty}$-representation $T_B(P(W), \phi_{P(W)})$.  This is an object of $\GRep_B(\Gamma_{\infty})$ because $T_B$ is a tensor exact functor (see \cite[3.4.1.6]{Broshi} or \cite[4.1.3]{LevinThesis}).  Similarly, one can define $\underline{M}_{G, B}$ which is quasi-inverse to $T_{G,B}$. Compatibility with extending the coefficients follows from \cite[Lemma 1.2.7(3)]{MFFGS}. 
\end{proof}    

\begin{defn} \label{defnheight} Let $B$ be any $\Zp$-algebra.  A \emph{Kisin module with bounded height} over $B$ is a finitely generated projective $\fS_B$-module $\fM_B$ together with an isomorphism $\phi_{\fM_B}:\phz^*(\fM_B)[1/E(u)] \cong \fM_B[1/E(u)]$. We say that $(\fM_B, \phi_{\fM_B})$ has \emph{height in $[a, b]$} if
$$
E(u)^{a} \fM_B \supset \phi_{\fM_B}(\phz^*(\fM_B)) \supset E(u)^{b} \fM_B
$$
as submodules of $\fM_B[1/E(u)]$.  
\end{defn}   

Let $\Mod^{\phz, \bh}_{\fS_B}$ (resp. $\Mod^{\phz, [a,b]}_{\fS_B}$) denote the category of Kisin modules with bounded height (resp. height in $[a,b]$) with morphisms being $\fS_B$-module maps respecting Frobenii. $\Mod^{\phz, [0,h]}_{\fS_B}$ is the usual category of Kisin modules with height $\leq h$ as in \cite{PHT, MFFGS, Fcrystals}.

\begin{exam} Let $\fS(1)$ be the Kisin module whose underlying module is $\fS$ and whose Frobenius is given by $c_0^{-1} E(u) \varphi_{\fS}$ where $E(0) = c_0 p$. For any $\Zp$-algebra, we define $\fS_B(1)$ by base change from $\Zp$ and define $\cO_{\cE, B}(1) := \fS_B(1) \otimes_{\fS_B} \cO_{\cE, B}$, an \'etale $(\cO_{\cE, B}, \phz)$-module.     
\end{exam}

In order to reduce to the effective case (height in $[0,h]$), it is often useful to ``twist'' by tensoring with $\fS_B(1)$.  For any $\fM_B \in \Mod^{\phz, \mathrm{bh}}_{\fS_B}$ and any $n \in \Z$, define $\fM_B(n)$ by $n$-fold tensor product with $\fS_B(1)$ (negative $n$ being tensoring with the dual). It is not hard to see that if $\fM_B \in \Mod^{\phz, [a,b]}_{\fS_B}$ then $\fM_B(n) \in \Mod^{\phz, [a + n,b + n]}_{\fS_B}$.

\begin{defn} \label{GKisinmodule} Let $B$ be any $\La$-algebra. A \emph{$G$-Kisin module} over $B$ is a pair $(\fP_B, \phi_{\fP_B})$ where $\fP_B$ is a $G$-bundle on $\fS_{B}$ and $\phi_{\fP_B}: \varphi^*(\fP_B)[1/E(u)] \cong \fP_B[1/E(u)]$ is an isomorphism of $G$-bundles. Denote the category of such objects by $\GMod^{\phz, \mathrm{bh}}_{\fS_B}$. 
\end{defn}

\begin{rmk} Unlike Kisin module for $\GL_n$, $G$-bundles do not have endomorphisms. Additionally, there is no reasonable notion of effective $G$-Kisin module. The Frobenius on a $G$-Kisin module is only ever defined after inverting $E(u)$. Later, we use auxiliary representations of $G$ to impose height conditions. 
\end{rmk}   

The category $\Mod^{\phz, \mathrm{bh}}_{\fS_B}$ is a tensor exact category where a sequence of Kisin modules 
$$
0 \ra \fM'_B \ra \fM_B \ra \fM''_B \ra 0
$$
is \emph{exact} if the underlying sequence of $\fS_B$-modules is exact. For any $W \in \sideset{^f}{_{\La}}\Rep (G)$, the pushout $(\fP_B(W), \phi_{\fP_B}(W))$ is a Kisin module with bounded height. Using Theorem \ref{catbundle}, one can interpret $\GMod^{\phz, \mathrm{bh}}_{\fS_B}$ as the category of faithful exact tensor functors from $\sideset{^f}{_{\La}}\Rep (G)$ to $\Mod^{\phz, \bh}_{\fS_B}$.

Since $E(u)$ is invertible in $\cO_{\cE}$, there is natural map $\fS_B[1/E(u)] \ra \cO_{\cE, B}$ for any $\Zp$-algebra $B$.  This induces a functor
$$
\Upsilon_G:\GMod^{\phz, \mathrm{bh}}_{\fS_B} \ra \GMod^{\phz}_{\cO_{\cE, B}}
$$
for any $\La$-algebra $B$.

\begin{defn} \label{GKisinlattice} Let $B$ be any $\La$-algebra and let $P_B \in \GMod^{\phz}_{\cO_{\cE, B}}$. A \emph{$G$-Kisin lattice of $P_B$} is a pair $(\fP_B, \alpha)$ where $\fP_B \in \GMod^{\phz, \mathrm{bh}}_{\fS_B}$ and $\alpha:\Upsilon_G(\fP_B) \cong P_B$ is an isomorphism. 
\end{defn}

From the Tannakian perspective, a $G$-Kisin lattice of $P$ is equivalent to Kisin lattices $\fM_W$ in $P(W)$ for each $W \in \sideset{^f}{_{\La}}\Rep (G)$ functorial in $W$ and compatible with tensor products. Furthermore, we have the following which says that the bounded height condition can be checked on a single faithful representation.

\begin{prop} \label{onerepheight} Let $P_B \in \GMod^{\phz}_{\cO_{\cE, B}}$. A $G$-Kisin lattice of $P_B$ is equivalent to an extension $\fP_B$ of the bundle $P_B$ to $\Spec \fS_B$ such that for a single faithful representation $V \in \sideset{^f}{_{\La}}\Rep (G)$,
$$
\fP_B(V) \subset P_B(V)
$$
is a Kisin lattice of bounded height.
\end{prop}
\begin{proof} The only claim which is does not follow from unwinding definitions is that if we have an extension $\fP_{B}$ such that $\fP_{B}(V) \subset P_B(V)$ is a Kisin lattice for a single faithful representation $V$, then $\fP_{B}(W) \subset P_B(W)$ is a Kisin lattice for all representations $W$ of $G$.

By \cite[C.1.7]{LevinThesis}, any $W \in \sideset{^f}{_{\La}}\Rep (G)$ can be written as a subquotient of direct sums of tensor products of $V$ and the dual of $V$.  It suffices then to prove that bounded height is stable under duals, tensor products, quotients, and saturated subrepresentations.

Duals and tensor products are easy to check.  For subquotients, let $0 \ra M_B \ra N_B \ra L_B \ra 0$ be an exact sequence of \'etale $(\cO_{\cE, B}, \phz)$-modules.  Suppose that the sequence is induced by an exact sequence
$$
0 \ra \fM_B \ra \fN_B \ra \fL_B \ra 0
$$
of projective $\fS_B$-lattices. Assume $\fN_B$ has bounded height with respect to $\phi_{N_B}$. By twisting, we can assume $\fN_B$ has height in $[0,h]$. 

Since $\fM_B = M_B \cap \fN_B$, $\fM_B$ is $\phi_{M_B}$-stable. Similarly, $\fL_B$ is $\phi_{L_B}$-stable. Consider the diagram
$$
\xymatrix{
0 \ar[r] & \phz^*(\fM_B) \ar[r] \ar[d]^{\phi_{M_B}} & \phz^*(\fN_B) \ar[r] \ar[d]^{\phi_{N_B}} & \phz^*(\fL_B) \ar[r] \ar[d]^{\phi_{L_B}} & 0 \\
0 \ar[r] & \fM_B \ar[r] & \fN_B \ar[r] & \fL_B \ar[r] & 0. \\
}
$$
All the linearizations are injective because they are isomorphisms at the level of $\cO_{\cE,B}$-modules. By the snake lemma, the sequence of cokernels is exact. If $E(u)^h$ kills $\Coker(\phi_{N_B})$, then it kills $\Coker(\phi_{M_B})$ and $\Coker(\phi_{P_B})$ as well.  Thus, $\fM_B$ and $\fP_B$ both have height in $[0,h]$ whenever $\fN_B$ does. 
\end{proof}

\begin{defn} For any $B$ as in Proposition \ref{normGequiv}, define
$$
T_{G, \fS_B}:\GMod^{\phz, \mathrm{bh}}_{\fS_B} \ra \GRep_B(\Gamma_{\infty})
$$
to be the composition $T_{G, \fS_B} := T_{G, B} \circ \Upsilon_G$.
\end{defn}


We end this section with an important full faithfulness result:
\begin{prop} \label{uniqueness} Assume $B$ is finite flat over $\La$. Then the natural extension map 
$$
\Upsilon_G:\GMod^{\phz, \mathrm{bh}}_{\fS_B} \rightarrow \GMod^{\phz}_{\cO_{\cE, B}}
$$
is fully faithful.
\end{prop}
\begin{proof} This follows from full faithfulness of $\Upsilon_{\GL_n}$ for all $n \geq 1$ by considering a faithful representation of $G$.  When $B = \Zp$, this is \cite[11.2.7]{PHT}.  One can reduce to this case by forgetting coefficients since any finitely generated projective $\fS_B$-module is finite free over $\fS$.
\end{proof}   

\subsection{Resolutions of $G$-valued deformations rings}

Fix a faithful representation $V$ of $G$ over $\La$ and integers $a, b$ with $a \leq b$.  We will use $V$ and $a, b$ to impose finiteness conditions on our moduli space.

\begin{defn}  Let $B$ be any $\La$-algebra. We say a $G$-Kisin lattice $\fP_{B}$ in $(P_{B}, \phi_{P_B}) \in \GMod^{\phz}_{\cO_{\cE, B}}$ has \emph{height in $[a,b]$} if $\fP_{B}(V)$ in $P_B(V)$ has height in $[a,b]$.
\end{defn}

For any finite local Artinian $\La$-algebra $A$ and any $(P_A, \phi_{P_A}) \in \GMod^{\phz}_{\cO_{\cE, A}}$, consider the following moduli problem over $\Spec A$: 
$$
X^{[a,b]}_{P_A}(B):= \{G \text{-Kisin lattices in }P_A \otimes_{\cO_{\cE, A}} \cO_{\cE, B} \text{ with height in }[a,b] \}/ \cong
$$
for any $A$-algebra $B$. 
 
\begin{thm} \label{GKisinproj} Assume that $P_A$ is a trivial bundle over $\Spec \cO_{\cE, A}$. The functor $X^{[a,b]}_{P_A}$ is represented by a closed finite type subscheme of the affine Grassmanian $\Gr_{G'}$ over $\Spec A$ where $G'$ is the Weil restriction $\Res_{(W \otimes_{\Zp} \La)/\La} G$.
\end{thm}
\begin{proof} By Proposition \ref{onerepheight}, $X^{[a,b]}_{P_A}(B)$ is the set of bundles over $\fS_B$ extending $P_B := P_A \otimes_{\cO_{\cE, A}} \cO_{\cE, B}$ with height in $[a,b]$ with respect to $V$. We want to identify this set with a subset of $\Gr_{G'}(B)$.  

Consider the following diagram
$$
\xymatrix{
\fS \otimes_{\Zp} B \ar[r] \ar[d] & (W \otimes_{\Zp} B)[\![u]\!] \ar[d] \\
\cO_{\cE, B} \ar[r] & (W \otimes_{\Zp} B)((u)),
}
$$
where the vertical arrows are localization at $u$ and the top horizontal arrow is $u$-adic completion. The Beauville-Laszlo gluing lemma (\ref{descent}) says that the set of extensions of $P_B$ to $\fS_B$ is in bijection with the set of extensions of $\widehat{P}_B$ to $W_B[\![u]\!]$, where $\widehat{P}_B$ is the $u$-adic completion.  This second set is in bijection with the $B$-points of the Weil restriction $\Res_{(W \otimes_{\Zp} \La)/\La} \Gr_G$ which is isomorphic to $\Gr_{G'}$ by \cite[Lemma 1.16]{AGRicharz} or \cite[3.4.2]{LevinThesis}. 

Set $M_A := P_A(V)$. By \cite[Proposition 1.3]{PST}, the functor $X_{M_A}^{[a,b]}$ of Kisin lattices in $M_A$ with height in $[a,b]$ is represented by a closed subscheme of $\Gr_{\Res_{(W \otimes_{\Zp} \La)/\La} \GL(V)}$. Evaluation at $V$ induces a map of functors,
\begin{equation} \label{m1}
X^{[a,b]}_{P_A} \ra X_{M_A}^{[a,b]}.
\end{equation}
The subset $X^{[a,b]}_{P_A}(B) \subset \Gr_{G'}(B)$ is exactly the preimage of $X^{[a,b]}_{M_A}(B)$ by Proposition \ref{onerepheight}.
\end{proof}

We now extend the construction beyond the Artinian setting by passing to the limit. Let $R$ be a complete local Noetherian $\La$-algebra with residue field $\F$.  Let $\eta:\Gamma_{\infty} \ra G(R)$ be a continuous representation. 

\begin{prop} \label{resolution} For any $n \geq 1$, let $\eta_n:\Gamma_{\infty} \ra G(R/m_R^n)$ denote the reduction mod $m_R^n$.  From $\{\eta_n\}$, we construct a system $\underline{M}_{G, R/m_R^n}(\eta_n) =: (P_{\eta_n}, \phi_n) \in \GMod^{\phz}_{\cO_{\cE, R/m_R^n}}$. Assume that $P_{\eta_1}$ is a trivial $G$-bundle. There exists a projective $R$-scheme 
$$
\Theta:X^{[a,b]}_{\eta} \ra \Spec R,
$$
whose reduction modulo $m_R^n$ is $X^{[a,b]}_{P_{\eta_n}}$ for any $n \geq 1$.
\end{prop}
\begin{proof} By Proposition \ref{normGequiv}, there are natural isomorphisms $P_{\eta_{n+1}} \otimes_{\cO_{\cE, R/m_R^{n+1}}} \cO_{\cE, R/m_R^{n}} \cong P_{\eta_n}$ for all $n \geq 1$.  Since $P_{\eta_1}$ is a trivial $G$-bundle, all $P_{\eta_n}$ are trivial by Proposition \ref{Gtriv} (1) so we can apply Theorem \ref{GKisinproj}. Consider then the system $\{ X^{[a,b]}_{P_{\eta_n}} \}$ of schemes over $\{ R/m_R^n \}$. Since $G'$ is reductive, the affine Grassmanian $\Gr_{G'}$ is ind-projective (\cite[Theorem 3.3.11]{LevinThesis}). In particular, any ample line bundle on $\Gr_{G'}$ will restrict to a compatible system of ample line bundles on
$$
\{ X^{[a,b]}_{P_{\eta_n}} \}.
$$
By formal GAGA (EGA $\mathrm{III}_1$ 5.4.5), there exists a projective $R$-scheme $X^{[a,b]}_{\eta}$ whose reductions modulo $m_R^n$ are $X^{[a,b]}_{P_{\eta_n}}$.
\end{proof} 

\begin{rmk} Unlike for $\GL_n$, there are non-trivial $G$-bundles over $\Spec \F(\!(u)\!)$ which is why we need the assumption in Proposition \ref{resolution}.  If $P_{\eta_1}$ admits any $G$-Kisin lattice $\fP_{\eta_1}$, by Proposition \ref{Gtriv} (2), the $G$-bundle $\fP_{\eta_1}$ is trivial since $\fS_{\F}$ is a semi-local ring with finite residue fields.  Thus, the assumption in Proposition \ref{resolution} is natural if you are interested in studying $\Gamma_{\infty}$-representations of finite height.  By Steinberg's Theorem, one can always make $P_{\eta_1}$ trivial by passing to finite extension $\F'$ of $\F$. 
\end{rmk}     

We record for reference the following compatibility with base change:
\begin{prop} \label{fhfunctoriality}  Let $f:R \ra S$ be a local map of complete local Noetherian $\La$-algebras with finite residue fields of characteristic $p$. Let $\eta_S$ be the induced map $\Gamma_{\infty} \ra G(S)$.  Then, there is a natural map $f':X^{[a,b]}_{\eta_S} \ra X^{[a,b]}_{\eta}$ which makes the following diagram Cartesian:
$$
\xymatrix{
X^{[a,b]}_{\eta_S} \ar[r]^{f'} \ar[d] &  X^{[a,b]}_{\eta} \ar[d] \\
\Spec S \ar[r]^{f} & \Spec R. \\
}
$$
In particular, if $R \ra S$ is surjective, then $f'$ is a closed immersion. 
\end{prop} 

We will now study the projective $F$-morphism
$$
\Theta[1/p]: X^{[a,b]}_{\eta}[1/p] \ra \Spec R[1/p].
$$
We show it is a \emph{closed immersion} (this is essentially a consequence of \ref{uniqueness}) and that the closed points of the image are $G$-valued representations with height in $[a,b]$ in a suitable sense (\ref{fhimage}). Next, we show that if $\eta$ is the restriction of $\eta':\Gamma_K \ra G(R)$, then the image of $\Theta[1/p]$ contains all semi-stable representations with $\eta'(V)$ having Hodge-Tate weights in $[a,b]$. These are generalizations of results from \cite{PST}.

The following lemma will be useful at several key points:
\begin{lemma}[Extension Lemma] \label{extensionlemma}  Let $G$ be a smooth affine group scheme over $\La$. Let $C$ be a finite flat $\La$-algebra and let $U$ be the open complement of the finite set of closed points of $\Spec \fS_C$.
\begin{enumerate} [$(1)$]
\item There is an equivalence of categories between $G$-bundles $Q$ on $U$ and the category of triples $(\fP^*, P, \gamma)$ where $\fP^*$ is a $G$-bundle on $\Spec \fS_C[1/p]$, $P$ is $G$-bundle $\Spec \cO_{\cE, C}$, and $\gamma$ is an isomorphism of their restrictions to $\Spec \cO_{\cE, C}[1/p]$.
\item Assume $G$ is a reductive group scheme with connected fibers. Let $V$ be a faithful representation of $G$ over $\La$. If $Q$ be a $G$-bundle on $U$ such that the locally free coherent sheaf $Q(V)$ on $U$ extends to a projective $\fS_C$-module $\fM_C$, then there exists a unique $($up to unique isomorphism$)$ $G$-bundle $\widetilde{Q}$ over $\Spec \fS_C$ such that $\widetilde{Q}|_U \cong Q$ and $\widetilde{Q}(V) = \fM_C$. 
\end{enumerate}
\end{lemma}
\begin{proof}
First note that we can write $U$ as the union of $\Spec \fS_C[1/u]$ and $\Spec \fS_C[1/p]$. Recall also that $\cO_{\cE, C}$ is the $p$-adic completion of $\fS_C[1/u]$. Since $p$ is non-zero divisor in $\fS_C[1/u]$, we can apply the gluing lemma (\ref{descent}) to $P$ and $\fP^*[1/u]$ to construct a $G$-bundle $Q'$ on $\Spec \fS_C[1/u]$ which by construction is isomorphic to $\fP^*$ along $\Spec \fS_C[1/u, 1/p]$. The $G$-bundles $\fP^*$ and $Q'$ glue to give bundle $Q$ over $U$. Each step in the construction is a categorical equivalence. 

For part (2), consider the functor $|\text{Fib}_{\fM_C}|$ which by Lemma \ref{changeofgroup} and \cite[Theorem C.2.5]{LevinThesis} is represented by an affine scheme $Y$. $\fM_C$ defines a $U$-point of $\text{Fib}_{\fM_C}$.  Since $\Gamma(U, \cO_U) = \fS_C$, we deduce that 
$$
\Hom_{\fS_C}(\Spec \fS_C, \text{Fib}_{\fM_C}) = \Hom_{\fS_C}(U, \text{Fib}_{\fM_C}).
$$ 
A $\fS_C$-point of $\text{Fib}_{\fM_C}$ is exactly a bundle $\widetilde{Q}$ extending $Q$ and mapping to $\fM_C$. 

A similar argument, using that the Isom-scheme between $G$-bundles is representable by an affine scheme, shows that if an extension exists it is unique up to unique isomorphism (without any reductivity hypotheses).  
\end{proof} 

Let $B$ be any finite local $F$-algebra with residue field $F'$. Define $B^0$ to be the subring of elements which map to $\cO_{F'}$ modulo the maximal ideal of $B$. Let $\text{Int}_B$ denote the set of finitely generated $\cO_{F'}$-subalgebras $C$ of $B^0$ such that $C[1/p] = B$. 

\begin{defn} \label{defnbh} A continuous homomorphism $\eta:\Gamma_{\infty} \ra G(B)$ has \emph{bounded height} if there exists a $C \in \text{Int}_B$ and $g \in G(B)$ such that
\begin{enumerate} [$(1)$]
\item $\eta'_C := g \eta g^{-1}$ factors through $G(C)$; 
\item $\underline{M}_{G,C}(\eta'_C) \in \GMod^{\phz}_{\cO_{\cE, C}}$ admits a $G$-Kisin lattice of bounded height.
\end{enumerate}  
We define \emph{height in $[a,b]$} with respect to the chosen faithful representation $V$ by replacing bounded height in (2) with height in $[a,b]$.
\end{defn}

\begin{lemma} \label{fhwithcoeff} Let $B$ be a finite local $\Qp$-algebra and choose $C \in \Int_B$ and $M_C \in \Mod^{\varphi, \et}_{\cO_{\cE, C}}$.  If $M_C$ considered as an $\cO_{\cE}$-module has bounded height $($resp. height in $[a,b])$, then there exists some $C' \supset C$ in $\Int_B$, such that $M_C \otimes_C C'$ has bounded height $($resp. height in $[a,b])$.  
\end{lemma}
\begin{proof}  This is the main content in the proof of part (2) of Proposition 1.6.4 in \cite{PST}. If $F'$ is the residue field of $B$, then one first constructs a Kisin lattice $\fM_{\cO_{F'}}$ in $M_C \otimes_C \cO_{F'}$. The Kisin lattice in $M_C \otimes_C {C'}$ is constructed by lifting $\fM_{\cO_{F'}}$ (the extension to $C'$ is required to insure that the lift is $\phi$-stable).  
\end{proof}

\begin{prop} \label{fhimage} The morphism $\Theta$ becomes a closed immersion after inverting $p$.  Furthermore, if $\Spec R_{\eta}^{[a,b]} \subset \Spec R$ is the scheme-theoretic image of $\Theta$, then for any finite $F$-algebra $B$, a $\La$-algebra map $x:R \ra B$ factors through $R_{\eta}^{[a,b]}$ if and only if $\eta \otimes_{R, x} B$ has height in $[a,b]$. 
\end{prop}
\begin{proof} The map $\Theta$ is injective on $C$-points for any finite flat $\La$-algebra $C$ by Proposition \ref{uniqueness}.  The proof of the first assertion is then the same as in \cite[Proposition 1.6.4]{PST}. 

For the second assertion, say $x:R \ra B$ factors through the $R_{\eta}^{[a,b]}$. Because $\Theta[1/p]$ is a closed immersion, $x:R \ra B$ comes from a $B$-point $y$ of $X^{[a,b]}_{\eta}$.  Any such $x$ is induced by $x_C:R \ra C$ for some $C \in \Int_B$.   By properness of $\Theta$, there exists $y_C \in X^{[a,b]}_{\eta}(C)$ such that $\Theta(y_C) = x_C$.  This implies that $\eta \otimes_{R, x_C} C$ has height in $[a,b]$ as a $G$-valued representation and hence $\eta \otimes_{R, x} B$ also has height in $[a,b]$ (see Definition \ref{defnbh}).

Now, let $x:R \ra B$ be a homomorphism such that $\eta_B := \eta \otimes_{R, x} B$ has height in $[a,b]$ as a $G$-valued representations. Any homomorphism $R \ra B$ factors through some $C \in \Int_B$ so that $\eta_B$ has image in $G(C)$; call this map $\eta_C$. We claim that there exist some $C' \supset C$ in $\Int_B$ such that $\eta_{C'} = \eta_C \otimes_C C'$ has height in $[a,b]$ and hence $x$ is in the image of $X^{[a,b]}_{\eta}(B)$. Essentially, we have to show that if one Galois stable ``lattice'' in $\eta_B$ has finite height then all ``lattices'' do.  For $\GL_n$, this is Lemma 2.1.15  in \cite{Fcrystals}.  We invoke the $\GL_n$ result below.     

Since $\eta_B$ has height in $[a,b]$, there exists $C' \in \Int_B$ and $g \in G(B)$ such that $\eta' = g \eta_B g^{-1}$ factors through $G(C')$ and has height in $[a,b]$. Enlarging $C$ if necessary, we assume both $\eta_C$ and $\eta'$ are valued in $G(C)$. Let $P_{\eta} := \underline{M}_{G, C} (\eta)$ and $P_{\eta'} := \underline{M}_{G, C}(\eta')$. Then $g$ induces an isomorphism
$$
P_{\eta'}[1/p] \cong P_{\eta_C}[1/p].
$$
Since $P_{\eta'}$ has a $G$-Kisin lattice with height in $[a,b]$, we get a bundle $\fQ_C$ over $\fS_{C}[1/p]$ extending $P_{\eta_C}[1/p]$. By Lemma \ref{extensionlemma}(1), $P_{\eta'}$ and $\fQ_C$ glue to give a bundle $Q_C$ over the complement of the closed points of $\Spec \fS_{C}$.

We would like to apply Lemma \ref{extensionlemma} (2). $P_{\eta_C}(V)$ has height in $[a,b]$ as an $\cO_{\cE}$-module by \cite[Lemma 2.1.15]{Fcrystals} since it corresponds to a lattice in $\eta_C(V)[1/p] \cong \eta'(V)[1/p]$.  By Lemma \ref{fhwithcoeff}, there exists $\widetilde{C} \supset C$ in $\Int_B$ such that $P_{\eta_C}(V) \otimes_C \widetilde{C}$ has height in $[a,b]$ as a $\cO_{\cE, \widetilde{C}}$-module.  Replace $C$ by $\widetilde{C}$. Then, if $\fM_C$ is the unique Kisin lattice in $P_{\eta_C}(V)$, we have
$$
\fM'_{C}[1/p] \cap P_{\eta_C}(V)   = \fM_C
$$
where $\fM'_C$ is the unique Kisin lattice in $P_{\eta'}(V)$.  This shows that $Q_C(V)$ extends across the closed points so we can apply Lemma \ref{extensionlemma} (2) to construct a $G$-Kisin lattice of $P_{\eta_C}$.
\end{proof}

Now, assume that $\eta$ is the restriction to $\Gamma_{\infty}$ of a continuous representation of $\Gamma_K$ which we continue to call $\eta$. Recall the definition of semi-stable for a $G$-valued representation:
\begin{defn} If $B$ is a finite $F$-algebra, a continuous representation $\eta_B:\Gamma_K \ra G_F(B)$ is \emph{semi-stable} (respectively \emph{crystalline}) if for all representations $W \in \Rep_F (G_F)$ the induced representation $\eta_B(W)$ on $W \otimes_{F} B$ is semi-stable (respectively crystalline).
\end{defn}

Note that because the semi-stable and crystalline conditions are stable under tensor products and subquotients, it suffices to check these conditions on a single faithful representation of $G_F$. 

\begin{rmk} \label{HTconv} Since we are working with covariant functors, our convention will be that the cyclotomic character has Hodge-Tate weight $-1$.  This is unfortunately opposite the convention in \cite{PST}.
\end{rmk}

The following Theorem generalizes \cite[Theorem 2.5.5]{PST}:
\begin{thm} \label{stlocus} Let $R$ be a complete local Noetherian $\La$-algebra with finite residue field and $\eta:\Gamma_K \ra G(R)$ a continuous representation. Given any $a, b$ integers with $a < b$, there exists a quotient $R^{[a,b], \st}_{\eta}$ $($resp. $R^{[a,b], \cris}_{\eta})$ of $R^{[a,b]}_{\eta}$ with the property that if $B$ is any finite $F$-algebra and $x:R \ra B$ a map of $\La$-algebras, then $x$ factors through $R^{[a,b],\st}_{\eta}$ $($resp. $R^{[a,b], \cris}_{\eta})$ if and only if $\eta_x:\Gamma_K \ra G(B)$ is semi-stable $($resp. crystalline$)$ and $\eta_x(V)$ has Hodge-Tate weights in $[a,b]$. 
\end{thm}

Since the semi-stable and crystalline properties can be checked on a single faithful representation, the quotients $R^{[a,b], \text{st}}_{\eta(V)}$ and $R^{[a,b], \cris}_{\eta(V)}$ of $R$ constructed by applying \cite[Theorem 2.5.5]{PST} to $\eta(V)$ satisfy the universal property in Theorem \ref{stlocus} with respect to maps $x:R \ra B$ where $B$ is a finite $F$-algebra. What remains is to show that $R^{[a,b], \text{st}}_{\eta} : = R^{[a,b], \text{st}}_{\eta(V)}$ is a quotient of $R^{[a,b]}_{\eta}$, i.e., that ``semi-stable implies finite height.'' 

\begin{prop} \label{Gssimpliesfh} Let $R$ and $\eta$ be as in $\ref{stlocus}$. For any map $x:R \ra B$ with $B$ a finite local $F$-algebra, if the representation $\eta_x$ is semi-stable and if $\eta_x(V)$ has Hodge-Tate weights in $[a,b]$ then $x$ factors through $R_{\eta}^{[a,b]}$.
\end{prop}
\begin{proof}  By Lemma \ref{fhwithcoeff}, there exists $C \in \Int_B$ such that $\eta_x$ factors through $\GL(V_{C})$, hence $G(C)$, and that $M_{C} := P_{\eta_x}(V)$ admits a Kisin lattice $\fM_C$ with height in $[a,b]$. It suffices by \ref{onerepheight} to extend the bundle $P_{\eta_x}$ to $\Spec \fS_C$ such that $\fP_{\eta_x}(V) = \fM_{C}$.  

We will apply Lemma \ref{extensionlemma}.  Consider a candidate fiber functor $\fF_{\eta_x}$ for $\fP_{\eta_x}$ which assigns to any $W \in \sideset{^f}{_{\La}}\Rep (G)$ the unique Kisin lattice of bounded height in $\fM_W \subset P_{\eta_x}(W) = M_W$ (as an $\cO_{\cE}$-module not as an $\cO_{\cE, C'}$-module). Such a lattice exists since $\eta_x(W)$ is semi-stable. The difficulties are that $\fM_W$ may not be $\cO_{\cE, C'}$-projective and that it is not obvious whether $\fF_{\eta_x}$ is exact. It can happen that a non-exact sequence of $\fS$-module can map under $T_{\fS}$ to an exact sequence of $\Gamma_{\infty}$-representations (see \cite[Example 2.5.6]{LiuLattices}).

Let $B = C[1/p]$. By \cite[Corollary 1.6.3]{PST}, $\fM_W[1/p]$ is finite projective over $\fS_{C}[1/p] = \fS_B$ for all $W$. We claim furthermore that $\fF_{\eta_x} \otimes_{\fS_{C}} \fS_B$ is exact. For any exact sequence $0 \ra W'' \ra W \ra W' \ra 0$ in $\sideset{^f}{_{\La}}\Rep (G)$, we have a left-exact sequence
$$
0 \ra \fM_{W''}[1/p] \ra \fM_W[1/p] \ra \fM_{W'}[1/p].
$$
Exactness on the right follows from \cite[Lemma 4.2.22]{LevinThesis} on the behavior of exactness for sequences of $\fS$-modules.  Thus, $\fF_{\eta_x} \otimes_{\fS_{C}} \fS_B$ defines a bundle $\fP_B$ over $\fS_B$. Clearly, $\fP_B \otimes_{\fS_B} \cO_{\cE, B} \cong P_{\eta_x} \otimes_{\cO_{\cE, C}} \cO_{\cE, B}$. By Lemma \ref{extensionlemma}(1), we get a bundle $Q$ over $U$ such that $Q(W) = \fM_W|_U$. Since $\fM_V$ is a projective $\fS_{C}$-module by our choice of $C$, $Q$ extends to a bundle $\widetilde{Q}$ over $\fS_{C}$ by Lemma \ref{extensionlemma}(2).
\end{proof}

\subsection{Universal $G$-Kisin module and filtrations}

For this section, we make a small change in notation. Let $R_0$ be a complete local Noetherian $\La$-algebra with finite residue field and let $R = R_0[1/p]$. 

Define $\widehat{\fS}_{R_0}$ to be the $m_{R_0}$-adic completion of $\fS \otimes_{\Zp} R_0$. The Frobenius on $\fS \otimes_{\Zp} R_0$ extends to a Frobenius on $\widehat{\fS}_{R_0}$.

\begin{defn} A $(\widehat{\fS}_{R_0}[1/p], \phz)$-module of \emph{bounded height} is a finitely generated projective $\widehat{\fS}_{R_0}[1/p]$-module $\fM_R$ together with an isomorphism $$\phi_{R}:\phz^*(\fM_R)[1/E(u)] \cong \fM_R[1/E(u)].$$
\end{defn}

Let $\eta:\Gamma_{\infty} \ra G(R_0)$ be continuous representation. If $\widehat{\cO}_{\cE, R_0}$ is the $m_{R_0}$-adic completion of $\cO_{\cE, R_0}$, then the inverse limit of $\varprojlim \underline{M}_{G, R_0/m_{R_0}^n} (\eta_n)$ defines a pair $(P_{\eta}, \phi_{\eta})$ over $\widehat{\cO}_{\cE, R_0}$ (\cite[Corollary 2.3.5]{LevinThesis}).  Assume $R_0 = R_{0, \eta}^{[a,b]}$.  For any finite $F$-algebra $B$ and any homomorphism $x:R_0 \ra B$, there is a unique $G$-Kisin lattice in $P_{\eta} \otimes_{\widehat{\cO}_{\cE, R_0}, x} \cO_{\cE, B}$ (\ref{uniqueness}), call it $(\fP_x, \phi_x)$.  In the following theorem, we construct a universal $G$-bundle over $\widehat{\fS}_{R_0}[1/p]$ with a Frobenius which specializes to $(\fP_x, \phi_x)$ at every $x$.

\begin{thm} \label{universalKisin} Assume that $R_0 = R_{0, \eta}^{[a,b]}$. Let $B$ be a finite $F$-algebra. The pair $(P_{\eta}[1/p], \phi_{\eta}[1/p])$ extends to a $G$-bundle $\widetilde{\fP}_{\eta}$ over $\widehat{\fS}_{R_0}[1/p]$ together with a Frobenius $\phi_{\widetilde{\fP}_{\eta}}:\phz^*(\widetilde{\fP}_{\eta})[1/E(u)] \cong \widetilde{\fP}_{\eta}[1/E(u)]$ such that for any $x:R_0[1/p] \ra B$, the base change 
$$
\left(\widetilde{\fP}_{\eta} \otimes_{\widehat{\fS}_{R_0}[1/p]} \fS_B, \phi_{\widetilde{\fP}_{\eta}} \otimes_{\widehat{\fS}_{R_0}[1/p, 1/E(u)]} \fS_B[1/E(u)]\right)
$$
is $(\fP_x, \phi_x)$.  
\end{thm}
\begin{proof}

Let $X_n := X_{\eta_n}^{[a,b]}$ be the projective $R_0/m_{R_0}^n$-scheme as in \S 4.3.  Take $Y_n := X_n \times_{\Spec R_0/m_{R_0}^n} \Spec \fS_{R_0/m_{R_0}^n}$, a projective $\fS_{R_0/m_{R_0}^n}$-scheme. Let $X_{\eta}^{[a,b]} \ra \Spec R_0$ be the algebraization of $\varprojlim X_n$ as before. The base change $Y$ of $X_{\eta}^{[a,b]}$ along the map $R_0 \ra \widehat{\fS}_{R_0}$ has the property that
$$
Y \mod m_{R_0}^n \cong Y_n.
$$      
Furthermore, $Y$ is a proper $\widehat{\fS}_{R_0}$-scheme.  

Over each $Y_n$, we have a universal $G$-Kisin lattice $(\fP_n, \phi_n)$ with height in $[a,b]$.  By \cite[Corollary 2.3.5]{LevinThesis}, there exists a $G$-bundle $\fP_{\eta}$ on $Y$ such that $\fP_{\eta} \mod m_{R_0}^n = \fP_n$. We would like to construct a Frobenius $\phi$ over $Y[1/E(u)]$ which reduces to $\phi_n$ modulo $m_R^n$ for each $n \geq 1$. A priori, the Frobenius is only defined over the $m_{R_0}$-adic completion of $\widehat{\fS}_{R_0}[1/E(u)]$ which we denote by $\widehat{S}$.  

We have a projective morphism
$$
Y_{\widehat{S}} \ra \Spec \widehat{S},
$$
where $Y_{\widehat{S}}$ is the base change of $Y[1/E(u)]$ along $\Spec \widehat{S} \ra \Spec \widehat{\fS}_{R_0}[1/E(u)]$. $Y_{\widehat{S}}$ is faithfully flat over $Y[1/E(u)]$ since $\widehat{\fS}_{R_0}[1/E(u)]$ is Noetherian. Let $\Isom_G := \Isom_G(\phz^*(\fP_{\eta}), \fP_{\eta})$ be the affine finite type $Y$-scheme of $G$-bundle isomorphisms. The compatible system $\{ \phi_n \}$ lifts to an element 
$$
\widehat{\phi} \in \Isom_G(Y_{\widehat{S}}).
$$
We would like to descend $\widehat{\phi}$ to a $Y[1/E(u)]$-point of $\Isom_G$.  Let $i:G \iarrow \GL(V)$ be our chosen faithful representation. Consider the closed immersion
$$
i_*:\Isom_G \iarrow \Isom_{\GL(V)}(\phz^*(\fP_{\eta})(V), \fP_{\eta}(V)).
$$
The image $i_*(\widehat{\phi})$ descends to a $Y[1/E(u)]$-point of $\Isom_{\GL(V)}(\phz^*(\fP_{\eta})(V), \fP_{\eta}(V))$ (twist to reduce to the effective case).  Since $Y_{\widehat{S}}$ is faithfully flat over $Y[1/E(u)]$, for any closed immersion $Z \subset Z'$ of $Y$-schemes, we have 
$$
Z(Y[1/E(u)]) = Z(Y_{\widehat{S}}) \cap Z'(Y[1/E(u)]).
$$ 
Applying this with $Z' = \Isom_G$ and $Z = \Isom_{\GL(V)}(\phz^*(\fP_{\eta})(V), \fP_{\eta}(V))$, we get a universal pair $(\fP_{\eta}, \phi_{\eta})$ over $Y$ respectively $Y[1/E(u)]$. Since $R_0 = R_{0, \eta}^{[a,b]}$, $\Theta[1/p]: X^{[a,b]}_{\eta}[1/p] \ra R_0[1/p]$ is an isomorphism and the pair $\widetilde{\fP}_{\eta} := \fP_{\eta}[1/p]$ and $\phi_{\fP_{\eta}}[1/p]$ over $\widehat{\fS}_{R_0}[1/p]$ has the desired properties.
\end{proof}

We now discuss the notion of $p$-adic Hodge type for $G$-valued representation and relate this to a filtration associated to a $G$-Kisin module.

Let $B$ be any finite $F$-algebra. For any representation of $\Gamma_K$ on a finite free $B$-module $V_B$, set 
$$
D_{\dR}(V_B) := (V_B \otimes_{\Qp} B_{\dR})^{\Gamma_K},
$$
a filtered $(K \otimes_{\Qp} B)$-module whose associated graded is projective (see \cite[3.1.6, 3.2.2]{Balaji}).  Furthermore, $D_{\dR}$ defines a tensor exact functor from the category of de Rham representations on projective $B$-modules to the category $\Fil_{K \otimes_{\Qp} B}$  of filtered $(K \otimes_{\Qp} B)$-modules (see \cite[3.2.2]{Balaji}). For any field $\kappa$, $\Fil_{\kappa}$ will be the tensor category of $\Z$-filtered vectors spaces $(V, \{ \Fil^{i} V \})$ where $\Fil^i(V) \supset \Fil^{i+1} (V)$.  

We recall a few facts from the Tannakian theory of filtrations:
\begin{defn} Let $H$ be any reductive group over a field $\kappa$.  For any extension $\kappa' \supset \kappa$,  an \emph{$H$-filtration over $\kappa'$} is a tensor exact functor from $\Rep_{\kappa}(H)$ to $\Fil_{\kappa'}$.  
\end{defn}

Associated to any cocharacter $\nu:\Gm \ra H_{\kappa'}$ is a tensor exact functor from $\Rep_\kappa(H)$ to graded $\kappa'$-vector spaces which assigns to each representation $W$ the vector space $W_{\kappa'}$ with its weight grading defined by the $\Gm$-action through $\nu$ which we denote $\omega_{\nu}$ (see \cite[Example 2.30]{DM}).

\begin{defn} For any $H$-filtration $\cF$ over $\kappa'$, a \emph{splitting} of $\cF$ is an isomorphism between the $\gr(\cF)$ and $\omega_{\nu}$ for some $\nu:\Gm \ra H_{\kappa'}$. 
\end{defn} 

By \cite[Proposition IV.2.2.5]{Saavedra}, all $H$-filtrations over $\kappa'$ are splittable.  For any given $\cF$, the cocharacters $\nu$ for which there exists an isomorphism $\gr(\cF) \cong \omega_{\nu}$ lie in the common $H(\kappa')$-conjugacy class.  If $\kappa'$ is a finite extension of $\kappa$ contained in $\overline{\kappa}$, then the \emph{type} $[\nu_{\cF}]$ of the filtration $\cF$ is the geometric conjugacy class of $\nu$ for any splitting $\omega_{\nu}$ over $\kappa'$. For any conjugacy class $[\nu]$ of geometric cocharacters of $H$, there is a smallest field of definition contained in a chosen separable closure of $\kappa$ called the \emph{reflex field} of $[\nu]$. We denote this by $\kappa_{[\nu]}$ .

Let $G$ be as before so that $G_F$ is a (connected) reductive group over $F$, and  let $\eta:\Gamma_K \ra G(B)$ be a continuous representation which is de Rham. Then, $D_{\dR}$ defines a tensor exact functor from $\Rep_F(G_F)$ to $\Fil_{K \otimes_{\Qp} B}$ (see Proposition 3.2.2 in \cite{Balaji}) which we denote by $\cF^{\dR}_{\eta}$.

Fix a geometric cocharacter $\mu \in X_*((\Res_{(K \otimes_{\Qp} F)/F} G)_{\overline{F}})$ and denote its conjugacy class by $[\mu]$. The cocharacter $\mu$ is equivalent to a set $(\mu_{\psi})_{\psi:K \ra \overline{F}}$ of cocharacters $\mu_{\psi}$ of $G_{\overline{F}}$ indexed by $\Qp$-embeddings of $K$ into $\overline{F}.$

\begin{defn} \label{defnHtype} Let $F_{[\mu]}$ be the reflex field of $[\mu]$. For any embedding $\psi:K \ra \overline{F}$ over $\Qp$, let $\pr_{\psi}:K \otimes_{\Qp} \overline{F} \ra \overline{F}$ denote the projection. If $F'$ is a finite extension of $F_{[\mu]}$, a $G$-filtration $\cF$ over $K \otimes_{\Qp} F'$ has \emph{type} $[\mu]$ if $\pr_{\psi}^*(\cF \otimes_{F', i} \overline{F})$ has type $[\mu_{\psi}]$ for any $F_{[\mu]}$-embedding $i:F' \iarrow \overline{F}$. A de Rham representation $\eta:\Gamma_K \ra G(F')$ has \emph{$p$-adic Hodge type} $\mu$ if $\cF^{\dR}_{\eta}$ has type $[\mu]$.
\end{defn}  

Let $\La_{[\mu]}$ denote the ring of integers of $F_{[\mu]}$. For any $\mu$ in the conjugacy class $[\mu]$, $\Gm$ acts on $V \otimes_{\La} \overline{F}$ through $\mu_{\psi}$ for each $\psi:K \ra \overline{F}$.  We take $a$ and $b$ be the minimal and maximal weights taken over all $\mu_{\psi}$.       
    
\begin{thm} \label{ssmutype} Let $R_0$ be a complete local Noetherian $\La_{[\mu]}$-algebra with finite residue field and $\eta:\Gamma_K \ra G(R_0)$ a continuous homomorphism. Let $R^{[a,b], \st}_{0, \eta}$ be as in $\ref{stlocus}.$  There exists a quotient $R^{\st, \mu}_{0, \eta}$ of $R^{[a,b], \st}_{0, \eta}$ such that for any finite extension $F'$ of $F_{[\mu]}$, a homomorphism $\zeta:R_0 \ra F'$ factors through $R^{\st, \mu}_{0, \eta}$ if and only if the $G(F')$-valued representation corresponding to $\zeta$ is semi-stable with $p$-adic Hodge type $[\mu]$.
\end{thm}
\begin{proof} See \cite[4.0.9]{Balaji}.
\end{proof}   

\begin{rmk} One can deduce from the construction in \cite[4.0.9]{Balaji} or by other arguments (\cite[Theorem 6.1.19]{LevinThesis}) that the $p$-adic Hodge type on  
the generic fiber of the semi-stable deformation ring $R^{[a,b], \st}_{0, \eta}$ is locally constant so that $\Spec R^{\st, \mu}_{0, \eta}[1/p]$ is a union of connected components of $\Spec R^{[a,b], \st}_{0, \eta}[1/p].$ 
\end{rmk}

Finally, we recall how the de Rham filtration is obtained from the Kisin module.

\begin{defn} \label{filtrations} Let $B$ be a finite $\Qp$-algebra. Let $(\fM_B, \phi_B)$ be a Kisin module over $B$ with bounded height. Define 
$$
\Fil^i(\phz^*(\fM_B)) := \phi_B^{-1}(E(u)^i \fM_B) \cap \phz^*(\fM_B).
$$
Set $\fD_B := \phz^*(\fM_B)/E(u) \phz^*(\fM_B)$, a finite projective $(K \otimes_{\Qp} B)$-module. Define $\Fil^i(\fD_B)$ to be the image of $\Fil^i(\phz^*(\fM_B))$ in $\fD_B$.  
\end{defn}  

\begin{prop} \label{DRfiltcomparison} Let $B$ be a finite $\Qp$-algebra and let $V_B$ be a finite-free $B$-module with an action of $\Gamma_K$ which is semi-stable with Hodge-Tate weights in $[a,b]$.  Any $\Zp$-stable lattice in $V_B$ has finite height. If $\fM_B$ is the $(\fS_B, \phz)$-module of bounded height attached to $V_B$, then there is a natural isomorphism $\fD_B \cong D_{\dR}(V_B)$ of filtered $(K \otimes_{\Qp} B)$-modules.
\end{prop}
\begin{proof}
The relevant results are in the proof of Corollary 2.6.2 and Theorem 2.5.5(2) in \cite{PST}.  Since \cite{PST} works with contravariant functors, one has to do a small translation. Under the conventions of \cite{PST}, $\fM_B$ would be associated to the $B$-dual $V_B^*$ and it is shown there that $D_B \cong D_{\dR}^*(V_B^*)$ as filtered $K \otimes_{\Qp} B$-modules in the case where $[a,b] = [0,h]$. By compatibility with duality (\cite[Proposition 3.1.6]{Balaji}), $D_{\dR}^*(V_B^*) \cong D_{\dR}(V_B)$. The general case follows by twisting.   
\end{proof}

\section{Deformations of $G$-Kisin modules}

In this section, we study the local structure of the ``moduli space" of $G$-Kisin modules.  This generalizes results of \cite{MFFGS} and \cite{PRcoeff}.  $G$-Kisin modules may have non-trivial automorphisms and so it is more natural as was done in \cite[\S 2.2]{MFFGS} to work with \emph{groupoids}.  The goal of the section is to smoothly relate the deformation theory of a $G$-Kisin module to the local structure of a local model for the group $\Res_{(K \otimes_{\Qp} F)/F} G_F$.  

 Intuitively, the smooth modication (chain of formally smooth morphisms) corresponds to adding a trivialization to the $G$-Kisin module and then taking the ``image of Frobenius'' similar to Proposition 2.2.11 of \cite{MFFGS}. The target of the modification is a deformation functor for the moduli space $\Gr_G^{E(u), W}$ discussed in \S 3.3 which is a version of the affine Grassmanian which appears in the work of \cite{PZ} on local models. Finally, we show that the condition of having $p$-adic Hodge type $\mu$ is related to a (generalized) local model $M(\mu) \subset \Gr_G^{E(u), W}$.  In this section, there are no conditions on the cocharacter $\mu$.  We will impose conditions on $\mu$ only in the next section when we study the analogue of flat deformations. 

\subsection{Definitions and representability results}

Let $\F$ be the residue field of $\La$.  Define the categories
$$
\cC_{\La} = \{\text{Artin local } \La \text{-algebras with residue field } \F \}
$$
and
$$
\widehat{\cC}_{\La} = \{\text{complete local Noetherian } \La \text{-algebras with residue field } \F \}.
$$
Morphisms are local $\La$-algebra maps. Recall that fiber products in the category $\widehat{\cC}_{\La}$ exist and are represented by completed tensor products. A \emph{groupoid} over $\cC_{\La}$ (or $\widehat{\cC}_{\La}$) will be in the sense of Definition A.2.2 of \cite{MFFGS}; this is also known as a category cofibered in groupoids over $\cC_{\La}$ (or $\widehat{\cC}_{\La}$). Recall also the notion of a 2-fiber product of groupoids from (A.4) in \cite{MFFGS}. See Appendix \S 10 of \cite{Kim} for more details related to groupoids.

Choose a bounded height $G$-Kisin module $(\fP_{\F}, \phi_{\F}) \in \GMod^{\phz, \mathrm{bh}}_{\fS_{\F}}$. Define $D_{\fP_{\F}} = \cup_{a < b} D_{\fP_{\F}}^{[a,b]}$ to be the deformation groupoid of $\fP_{\F}$ as a $G$-Kisin module of bounded height over $\widehat{\cC}_{\La}$.  The morphisms $D_{\fP_{\F}}^{[a,b]} \subset D_{\fP_{\F}}$ are relatively representable closed immersions so intuitively $D_{\fP_{\F}}$ is an ind-object built out of the finite height pieces.  

Let $\cE^0$ denote the trivial $G$-bundle over $\La$. Throughout we will be choosing various trivializations of the $G$-bundle $\fP_{\F}$ and other related bundles.  This is always possible because $\fS_{\F}$ is a complete semi-local ring with all residue fields finite (see Proposition \ref{Gtriv} (2)).  

\begin{prop} \label{hullrep} For any $\fP_{\F}$ with height in $[a,b]$, the deformation groupoid $D_{\fP_{\F}}^{[a,b]}$ admits a formally smooth morphism $\pi:\Spf R \ra D_{\fP_{\F}}^{[a,b]}$ for some $R \in \widehat{\cC}_{\La}$ $($i.e., has a versal formal object in the sense of \cite{Rim}$)$.
\end{prop}
\begin{proof} One can check the abstract Schlessinger's criterion in \cite[Theorem 1.11]{Rim}.  However, it will be useful to have an explicit versal formal object. Fix a trivialization $\beta_{\F}$ of $\fP_{\F} \mod E(u)^N$ for any $N \geq 1$, and  define
$$
\widetilde{D}_{\fP_{\F}}^{[a,b], (N)}(A) := \lbrace (\fP_A, \beta_A) \mid \fP_A \in D^{[a,b]}_{\fP_{\F}}(A), \beta_A:\fP_A \cong \cE^0_{\fS_A} \! \! \mod E(u)^N \rbrace,
$$
where $\beta_A$ lifts $\beta_{\F}$. Since $G$ is smooth, the forgetful morphism $\pi^{(N)}:\widetilde{D}_{\fP_{\F}}^{[a,b], (N)} \ra D_{\fP_{\F}}^{[a,b]}$ is formally smooth for any $N$. 

If $N > \frac{b-a}{p-1}$, then $\widetilde{D}_{\fP_{\F}}^{[a,b], (N)}$ is pro-representable by a complete local Noetherian $\La$-algebra.  The proof uses Schlessinger's criterion.  The two key points are that objects in $\widetilde{D}_{\fP_{\F}}^{[a,b], (N)}(A)$ have no non-trivial automorphisms for which one inducts on the power of $p$ which kills $A$ (see \cite[Proposition 8.1.6]{LevinThesis}) and that the tangent space of the underlying functor is finite dimensional which uses a successive approximation argument (see \cite[Proposition 8.1.8]{LevinThesis}).
\end{proof}

It will also be useful to have an infinite version of $\widetilde{D}_{\fP_{\F}}^{[a,b], (N)}$.  Fix a trivialization $\beta_{\F}:\fP_{\F} \cong \cE^0_{\fS_{\F}}$. Define a groupoid on $\cC_{\La}$ by
$$
\widetilde{D}^{[a,b], (\infty)}_{\fP_{\F}}(A) := \lbrace (\fP_A, \beta_A) \mid \fP_A \in D^{[a,b]}_{\fP_{\F}}(A), \beta_A:\fP_A \cong \cE^0_{\fS_A} \rbrace,
$$
where $\beta_A$ lifts $\beta_{\F}$.  Define $\widetilde{D}^{(\infty)}_{\fP_{\F}} := \cup_{a < b} \widetilde{D}^{[a,b], (\infty)}_{\fP_{\F}}$. 

\subsection{Local models for Weil-restricted groups}

In this section, we associate to any geometric conjugacy class $[\mu]$ of cocharacters of $\Res_{(K \otimes_{\Qp} F)/F} G_F$ a local model $M(\mu)$ (Definition \ref{defnlocmodel}) over the ring of integers $\La_{[\mu]}$ of the reflex field $F_{[\mu]}$ of $[\mu]$ (the relevant parahoric here is $\Res_{(\cO_K \otimes_{\Zp} \La)/\La} G$). By construction, $M(\mu)$ is a flat projective $\La_{[\mu]}$-scheme. The principal result (Theorem \ref{locmodels}) says that $M(\mu)$ is normal and its special fiber is reduced. 

The details of the proof of Theorem \ref{locmodels} are in Chapter \S 10 of \cite{LevinThesis} where we follow the strategy introduced in \cite{PZ}. We cannot apply Pappas and Zhu's result directly because the group $\Res_{(K \otimes_{\Qp} F)/F} G_F$ usually does not split over a tame extension of $F$. In \cite{LevinLM}, we generalize \cite[\S 10]{LevinThesis} and \cite{PZ} to groups of the form $\Res_{L/F} H$ where $H$ is reductive group over $L$ which splits over a tame extension of $L$ and allow arbitrary parahoric level structure.  Here we recall the relevant definitions and results leaving the details to \cite{LevinThesis, LevinLM}. 

For any $\La$-algebra $R$, set $R_W := R \otimes_{\Zp} W$. Our local models are constructed inside the following moduli space:
\begin{defn} \label{BigAG} For any $\La$-algebra $R$, let $\widehat{R_W[u]}_{(E(u))}$ denote the $E(u)$-adic completion of $R_W[u]$.  Define
$$
\Gr^{E(u), W}_G(R) := \{ \text{isomorphism classes of pairs } (\cE, \alpha) \},
$$
where $\cE$ is a $G$-bundle on $\widehat{R_W[u]}_{(E(u))}$ and $\alpha:\cE|_{\widehat{R_W[u]}_{(E(u))}[E(u)^{-1}]} \cong \cE^0_{\widehat{R_W[u]}_{(E(u))}[E(u)^{-1}]}$.
\end{defn}

\begin{prop}  \label{AGfibers} The functor $\Gr^{E(u), W}_{G}$ is an ind-scheme which is ind-projective over $\La$.  Furthermore,
\begin{enumerate}
\item the generic fiber $\Gr^{E(u), W}_{G}[1/p]$ is naturally isomorphic to the affine Grassmanian of ${\Res_{(K \otimes_{\Qp} F)/F} G_F}$ over the field $F$; 
\item  if $k_0$ is the residue field of $W$, then the special fiber $\Gr^{E(u), W}_{G} \otimes_{\La} \F$ is naturally isomorphic to the affine Grassmanian of $\Res_{(k_0 \otimes_{\Fp} \F)/\F} (G_{\F})$.
\end{enumerate}
\end{prop}
\begin{proof} See \S 10.1 in \cite{LevinThesis}.
\end{proof}

Let $H$ be any reductive group over $F$ and $\Gr_H$ be the affine Grassmanian of $H$.  Associated to any geometric conjugacy class $[\mu]$ of cocharacters there is an affine Schubert variety $S(\mu)$ in $(\Gr_H)_{F_{[\mu]}}$ where $F_{[\mu]}$ is the reflex field of $[\mu]$.  These are the closures of orbits for the positive loop group $L^+H$. 

The geometric conjugacy classes of cocharacters of $H$ can be identified with the set of dominant cocharacters for a choice of maximal torus and Borel over $\overline{F}$.  The dominant cocharacters have partial ordering defined by $\mu \geq \la$ if and only if $\mu - \la$ is a non-negative sum of positive coroots.  Then, $S(\mu)_{\overline{F}}$ is then the union of the locally closed affine Schubert cells for all $\mu' \leq \mu$ (\cite[Proposition 2.8]{RicharzSV}).   

\begin{defn} \label{defnlocmodel} Let $F_{[\mu]}/F$ be the reflex field of $[\mu]$ with ring of integers $\La_{[\mu]}$. If $S(\mu) \subset \Gr_{\Res_{(K \otimes_{\Qp} F)/F} G_F} \otimes_F F_{[\mu]}$ is the closed affine Schubert variety associated to $\mu$, then the \emph{local model} $M(\mu)$ associated to $\mu$ is the flat closure of $S(\mu)$ in $\Gr^{E(u), W}_{G} \otimes_{\La} \La_{[\mu]}$. It is a flat projective scheme over $\Spec \La_{[\mu]}$.
\end{defn}

The main theorem on the geometry of local models is:
\begin{thm} \label{locmodels} Suppose that $p \nmid |\pi_1(G^{\mathrm{der}})|$ where $G^{\mathrm{der}}$ is the derived subgroup of $G$. Then $M(\mu)$ is normal.  The special fiber $M(\mu) \otimes_{\La_{\mu}} \overline{\F}$ is reduced, irreducible, normal, Cohen-Macaulay and Frobenius-split. 
\end{thm}

For the next subsection, it will useful to recall a group which acts on $\Gr^{E(u), W}_{G}$ and $M(\mu)$.  Define
$$
L^{+, E(u)} G (R)  := G (\widehat{R_W[u]}_{(E(u))}) = \varprojlim_{i \geq 1} G (R_W[u]/(E(u)^i))
$$
for all $\La$-algebras $R$.   $L^{+, E(u)} G$ is represented by a group scheme which is the projective limit of the affine flat finite type group schemes $\Res_{((\La \otimes_{\Zp} W)[u]/E(u)^i)/\La} G$.  

The group  $L^{+, E(u)} G$ acts on  $\Gr^{E(u), W}_{G}$ by changing the trivialization.  This action is \emph{nice} in the sense of \cite[A.3]{Gaitsgory}, i.e., $\Gr^{E(u), W}_{G} \cong \varinjlim_i Z_i$  where $Z_i$ are $L^{+, E(u)} G$-stable closed subschemes on which $L^{+, E(u)} G$ acts through the quotient $\Res_{((\La \otimes_{\Zp} W)[u]/E(u)^i)/\La} G$. 

\begin{cor} \label{stabilityofM} For any $\mu$, the local model $M(\mu)$ is stable under the action of $L^{+, E(u)} G$.
\end{cor}
\begin{proof} Since everything is flat, it suffices to show that $M(\mu)[1/p]$ is stable under $L^{+, E(u)} G[1/p]$. The functor $L^{+, E(u)} G [1/p]$ on $F$-algebras is naturally isomorphic to the positive loop group $L^+ \Res_{(K \otimes_{\Qp} F)/F}(G)$ such that the isomorphism in Proposition \ref{AGfibers}(1) is equivariant.  $M(\mu)[1/p]$ is the closed affine Schubert variety $S(\mu)$ which is stable under the action of this group. 
\end{proof}

\subsection{Smooth modification}

We begin by defining the deformation functor which will be the target of our modification. 

\begin{defn} \label{bardefn} Choose a $G$-bundle $Q_{\F}$ over $\fS_{\F}$ together with a trivialization $\delta_0$ of $Q_{\F}$ over $\fS_{\F}[1/E(u)]$. Define a deformation functor on $\cC_{\La}$ by
$$
\overline{D}_{Q_{\F}}(A) := \{ \text{isomorphism classes of triples } (\cE, \delta, \psi) \},
$$
where $\cE$ is a $G$-bundle on $\fS_A$, $\delta:\cE|_{\fS_A[E(u)^{-1}]} \cong \cE^0_{\fS_A[E(u)^{-1}]}$, and $\psi:\cE \otimes_{\fS_A} \fS_{\F} \cong Q_{\F}$ compatible with $\delta$ and $\delta_0$.
\end{defn}

\begin{exam} \label{barGLV} Let $G = \GL(V)$. For any $(Q_{A}, \delta_A) \in \overline{D}_{Q_{\F}}(A)$, $\delta_A$ identifies $Q_{A}$ with a ``lattice'' in $(V \otimes_{\La} \fS_A)[1/E(u)]$, i.e., a finitely generated projective $\fS_A$-module $L_{A}$ such that $L_{A}[1/E(u)] = (V \otimes_{\La} \fS_A)[1/E(u)]$.
\end{exam}

The main result of this section is the following:
\begin{thm} \label{Modification} Let $\La$ be a $\Zp$-finite flat local domain with residue field $\F$. Let $G$ be a connected reductive group over $\La$ and $\fP_{\F}$ a $G$-Kisin module with coefficients in $\F$. Fix a trivialization $\beta_{\F}$ of $\fP_{\F}$ as a $G$-bundle. There exists a diagram of groupoids over $\cC_{\La}$,
$$
\xymatrix{
& \widetilde{D}^{(\infty)}_{\fP_{\F}} \ar[dl]_{\pi^{(\infty)}} \ar[dr]^{\Psi} & \\
D_{\fP_{\F}} & & \overline{D}_{Q_{\F}}, \\
}
$$ 
where $Q_{\F} := (\phz^*(\fP_{\F}), \beta_{\F}[1/E(u)] \circ \phi_{\fP_{\F}})$. Both $\pi^{(\infty)}$ and $\Psi$ are formally smooth.
\end{thm}

Later in the section, we will refine this modification by imposing appropriate conditions on both sides. Intuitively, the above modification corresponds to adding a trivialization to the $G$-Kisin module and then taking the ``image of Frobenius.'' The groupoid $\widetilde{D}^{(\infty)}_{\fP_{\F}}$ is defined at the end of \S 3.1 and $\pi^{(\infty)}$ is formally smooth since $G$ is smooth.  Next, we construct the morphism $\Psi$ and show that it is formally smooth. To avoid excess notation, we sometimes omit the data of the residual isomorphisms modulo $m_A$.  One can check that the everything is compatible with such isomorphisms. 

\begin{defn} \label{defnPsi} For any $(\fP_A, \phi_{\fP_A}, \beta_A) \in \widetilde{D}^{(\infty)}_{\fP_{\F}}(A)$, we set 
$$
\Psi((\fP_A, \phi_{\fP_A}, \beta_A)) = (\phz^*(\fP_A), \delta_A),
$$
where $\delta_A$ is the composite
$$
\phz^*(\fP_A)[1/E(u)] \xrightarrow{\phi_{\fP_A}} \fP_A[1/E(u)] \xrightarrow{\beta_A[1/(E(u))]} \cE^0_{\fS_A} [1/E(u)].   
$$
\end{defn} 

\begin{prop}  The morphism $\Psi$ of groupoids is formally smooth.
\end{prop}
\begin{proof}
Choose $A \in \cC_{\La}$ and $I$ an ideal of $A$. Consider a pair $(Q_A, \delta_A) \in \overline{D}_{Q_{\F}}(A)$ over a pair $(Q_{A/I}, \delta_{A/I})$. Let $(\fP_{A/I}, \phi_{A/I}, \beta_{A/I})$ be an element in the fiber over $(Q_{A/I}, \delta_{A/I})$.  The triple $(\fP_{A/I}, \phi_{A/I}, \beta_{A/I})$ is isomorphic to a triple of the form $(\cE^0_{\fS_{A/I}}, \phi'_{A/I}, \text{Id}_{A/I})$.  Let $\gamma_{A/I}$ be the isomorphism between $\phz^*(\cE^0_{\fS_{A/I}})$ and $Q_{A/I}$. We want to construct a lift $(\fP_A, \phi_A, \beta_A)$ such that $\Psi(\fP_A, \phi_A, \beta_A) = (Q_A, \delta_A)$. Take $\fP_A = \cE^0_{\fS_A}$ to be the trivial bundle and $\beta_A$ to be the identity.  

Now, pick any lift $\gamma_A:\phz^*(\cE^0_{\fS_{A}}) \cong Q_{A}$ of $\gamma_{A/I}$ which exists since $G$ is smooth. We can define the Frobenius by
$$
\phi_{A} = \delta_A \circ \gamma_A[1/E(u)].
$$
It is easy to check that $\Psi(\fP_A, \phi_A, \beta_A) \cong (Q_A, \delta_A)$.      
\end{proof}

We would now like to relate $\overline{D}_{Q_{\F}}$ to $\Gr^{E(u), W}_G$ from the previous section.

\begin{prop} \label{idDbar} A pair $(Q_{\F}, \delta_0)$ as in Definition $\ref{bardefn}$ defines a point $x_{\F} \in \Gr_G^{E(u), W}(\F)$.  Furthermore, for any $A \in \cC_{\La}$, there is a natural functorial bijection between $\overline{D}_{Q_{\F}}(A)$ and the set of $x_A \in \Gr_G^{E(u), W}(A)$ such that $x_A \mod m_A = x_{\F}$.
\end{prop}
\begin{proof} Recall that $\fS_A = (W \otimes_{\Zp} A)[\![u]\!]$ because $A$ is finite over $\Zp$. $\Gr_G^{E(u), W}(A)$ is the set of isomorphism classes of bundles on the $E(u)$-adic completion of $(W \otimes_{\Zp} A)[u]$ together with a trivialization after inverting $E(u)$.  Since $p$ is nilpotent in $A$, we can identify $(W \otimes_{\Zp} A)[\![u]\!]$ and the $E(u)$-adic completion $\widehat{(W \otimes_{\Zp} A)[u]}_{(E(u))}$.  This identifies $\overline{D}_{Q_{\F}}(A)$ with the desired subset of $\Gr_G^{E(u), W}(A)$.  
\end{proof}

For any $\Zp$-algebra $A$, let $\widehat{S}_{A}$ denote the $E(u)$-adic completion of $(W \otimes_{\Zp} A)[u]$.

\begin{lemma} \label{switch}  For any finite flat $\Zp$-algebra $\La'$, there is a $(W \otimes_{\Zp} \La')[u]$-algebra isomorphism
$$
\fS_{\La'} \ra \widehat{S}_{\La'}. 
$$
\end{lemma}
\begin{proof} For any $n \geq 1$, we have an isomorphism
$$
\fS_{\La'}/p^n \cong \widehat{S}_{\La'}/p^n
$$
since $(E(u))$ and $u$ define the same adic topologies mod $p^n$. Passing to the limit, we get an isomorphism of their $p$-adic completions.  Both $\fS_{\La'}$ and $\widehat{S}_{\La'}$ are already $p$-adically complete and separated.
\end{proof} 

Fix a geometric cocharacter $\mu$ of $\Res_{(K \otimes_{\Qp} F)/F} G_F$; we can write $\mu = (\mu_{\psi})_{\psi:K \ra \overline{F}}$ where the $\mu_{\psi}$ are cocharacters of $G_{\overline{F}}$. \emph{Assume} that $F = F_{[\mu]}$ so that the generalized local model $M(\mu)$ is then a closed subscheme of $\Gr_G^{E(u), W}$ over $\La$ (\ref{defnlocmodel}). Recall that $V$ is a fixed faithful representation of $G$. For each $\psi$, $\mu_{\psi}$ induces an action of $\Gm$ on $V_{\overline{F}}$.  Define $a$ (resp. $b$) to the smallest  (resp. largest) weight appearing in $V_{\overline{F}}$ over all $\mu_{\psi}$.  
 
\begin{defn} \label{muDbar} Define a closed subfunctor $\overline{D}^{\mu}_{Q_{\F}}$ of $\overline{D}_{Q_{\F}}$ by 
$$
\overline{D}^{\mu}_{Q_{\F}}(A) := \{ (Q_A, \delta_A) \in \overline{D}_{Q_{\F}}(A) \mid (Q_A, \delta_A) \in M(\mu)(A) \}
$$
under the identification in Proposition \ref{idDbar}. Define $\widetilde{D}^{(\infty), \mu}_{\fP_{\F}}$ to be the base change of $\overline{D}^{\mu}_{\fP_{\F}}$ along $\Psi$.  It is a closed subgroupoid of $\widetilde{D}^{(\infty)}_{\fP_{\F}}$.
\end{defn}

The following proposition says that $\widetilde{D}^{(\infty), \mu}_{\fP_{\F}}$ descends to a closed subgroupoid $D^{\mu}_{\fP_{\F}}$ of $D_{\fP_{\F}}$:
\begin{prop} \label{mudescent} Let $a$ and $b$ be as in the discussion before Definition \ref{muDbar}. There is a closed subgroupoid $D^{\mu}_{\fP_{\F}} \subset D^{[a,b]}_{\fP_{\F}} \subset D_{\fP_{\F}}$ such that $\pi^{(\infty)}|_{\widetilde{D}^{(\infty), \mu}_{\fP_{\F}}}$ factors through $D^{\mu}_{\fP_{\F}}$ and 
$$
\widetilde{D}^{(\infty), \mu}_{\fP_{\F}} \ra D^{\mu}_{\fP_{\F}} \times_{D_{\fP_{\F}}} \widetilde{D}^{(\infty)}_{\fP_{\F}} 
$$
is an equivalence of closed subgroupoids. Furthermore, the map $\pi^{\mu}:\widetilde{D}^{(\infty), \mu}_{\fP_{\F}} \ra D^{\mu}_{\fP_{\F}}$ is formally smooth.    
\end{prop}
\begin{proof} For any $A \in \cC_{\La}$ define $D^{\mu}_{\fP_{\F}}(A)$ to be the full subcategory whose objects are $\pi^{(\infty)}(\widetilde{D}^{(\infty), \mu}_{\fP_{\F}}(A))$.  Observe that for any $x \in D^{\mu}_{\fP_{\F}}(A)$  the group $G(\fS_A)$ acts transitively on the fiber $(\pi^{(\infty)})^{-1}(x) \subset \widetilde{D}^{(\infty)}_{\fP_{\F}}(A)$ by changing the trivialization.  The key point is that by   Corollary \ref{stabilityofM}, $\widetilde{D}^{(\infty), \mu}_{\fP_{\F}}(A)$ is stable under $G(\fS_A)$.  Hence
\begin{equation} \label{z5}
(\pi^{(\infty)})^{-1}(x) \subset \widetilde{D}^{(\infty), \mu}_{\fP_{\F}}(A).
\end{equation}
It is not hard to see then  that the map to the fiber product is an isomorphism and that $\pi^{\mu}$ is formally smooth.

It remains to show that $D^{\mu}_{\fP_{\F}} \ra D_{\fP_{\F}}$ is closed. Let $\fP_A \in D_{\fP_{\F}}(A)$ and choose a trivialization $\beta_A$ of $\fP_A$, i.e., a lift to $\widetilde{D}^{(\infty)}_{\fP_{\F}}(A)$. We want a quotient $A \ra A'$ such that for any $f:A \ra B$, $\fP_A \otimes_{A,f} B \in D^{\mu}_{\fP_{\F}}(B)$ if and only if $f$ factors through $A'$.  Let $A \ra A'$ represent the closed condition $\widetilde{D}^{(\infty), \mu}_{\fP_{\F}} \subset \widetilde{D}^{(\infty)}_{\fP_{\F}}$.  Clearly, $\fP_A \otimes_A A' \in D^{\mu}_{\fP_{\F}}(A')$ and so any further base change is as well.  Now, let $f:A \ra B$ be such that $\fP_A \otimes_{A,f} B \in D^{\mu}_{\fP_{\F}}(B)$.  The trivialization $\beta_A$ induces a trivialization $\beta_B$ on $\fP_B$.  The pair $(\fP_B, \beta_B)$ lies in $\widetilde{D}^{(\infty), \mu}_{\fP_{\F}}(B)$  by (\ref{z5}).    
\end{proof}

We have constructed a diagram of formally smooth morphisms
\begin{equation} \label{musmoothmod}
\xymatrix{
& \widetilde{D}^{(\infty),\mu}_{\fP_{\F}} \ar[dl]_{\pi^{\mu}} \ar[dr]^{\Psi^{\mu}} & \\
D^{\mu}_{\fP_{\F}} & & \overline{D}^{\mu}_{Q_{\F}}, \\
}
\end{equation}
where $\overline{D}^{\mu}_{Q_{\F}}$ is represented by the completed local ring at the $\F$-point of $M(\mu)$ corresponding to $(Q_{\F}, \delta_{\F})$. Next, we would like to replace $\widetilde{D}^{(\infty), \mu}_{\fP_{\F}}$ by a ``smaller'' groupoid which is representable. 

Let $a, b$ be as in the discussion before Definition \ref{muDbar} and choose $N > b-a$.  Recall the representable groupoid $\widetilde{D}^{[a,b], (N)}_{\fP_{\F}}$ (Proposition \ref{hullrep}). Define a closed subgroupoid 
$$
\widetilde{D}^{(N), \mu}_{\fP_{\F}} := D^{\mu}_{\fP_{\F}} \times_{D_{\fP_{\F}}} \widetilde{D}^{[a,b], (N)}_{\fP_{\F}}
$$
of $\widetilde{D}^{[a,b], (N)}_{\fP_{\F}}$. By Proposition \ref{mudescent}, the morphism $\widetilde{D}^{(\infty), \mu}_{\fP_{\F}} \ra D^{(N), \mu}_{\fP_{\F}}$ is formally smooth. 

\begin{prop} \label{musmoothfinite} For any $N > b-a$, the morphism $\Psi^{\mu}:\widetilde{D}^{(\infty), \mu}_{\fP_{\F}} \ra \overline{D}^{\mu}_{\fP_{\F}}$ factors through $\widetilde{D}^{(N), \mu}_{\fP_{\F}}$.  Furthermore, $\widetilde{D}^{(N), \mu}_{\fP_{\F}}$ is formally smooth over $\overline{D}^{\mu}_{Q_{\F}}$. 
\end{prop}
\begin{proof} By our assumption on $N$, $\widetilde{D}^{(N), \mu}_{\fP_{\F}}$ is representable so it suffices to define the factorization $\Psi^{\mu}_N:\widetilde{D}^{(N), \mu}_{\fP_{\F}} \ra \overline{D}^{\mu}_{\fP_{\F}}$ on underlying functors. For any $x \in \widetilde{D}^{(N), \mu}_{\fP_{\F}}(A)$, set
$$
\Psi^{(N), \mu}(x) := \Psi^{\mu}(\widetilde{x})
$$
for any lift $\widetilde{x}$ of $x$ to $\widetilde{D}^{(\infty), \mu}_{\fP_{\F}}(A)$. The image is independent of the choice of lift by Corollary \ref{stabilityofM}. The map $\Psi^{(N),\mu}$ is formally smooth since $\Psi^{\mu}$ is.
\end{proof}

In the remainder of this section, we discuss the relationship between $D^{\mu}_{\fP_{\F}}$ and $p$-adic Hodge type $\mu$.  For this, it will useful to work in a larger category than $\widehat{\cC}_{\La}$.   All of our deformation problems can be extended to the category of complete local Noetherian $\La$-algebras $R$ with finite residue field.  For any such $R$, we define $D^{\star}_{\fP_{\F}}(R)$ (resp. $\widetilde{D}^{\star}_{\fP_{\F}}(R)$,  $\overline{D}^{\star}_{Q_{\F}}(R)$) to be the category of deformations to $R$ of $\fP_{\F} \otimes_{\F} R/m_R$ with condition $\star$, where $\star$ is any of our various conditions.  For any finite local $\La$-algebra $\La'$, the category $\widehat{\cC}_{\La'}$ is a subcategory of the category of complete local Noetherian $\La$-algebras with finite residue field. 

The functors $\widetilde{D}^{[a,b], (N)}_{\fP_{\F}}, \widetilde{D}^{(N), \mu}_{\fP_{\F}}$ and $\overline{D}^{\mu}_{Q_{\F}}$ are all representable on $\widehat{\cC}_{\La}$.  It is easy to check using the criterion in \cite[Proposition 1.4.3.6]{CMbook} that these functors commute with change in coefficients, i.e., if $\widetilde{R}^{[a,b], (N)}$ represents $\widetilde{D}^{[a,b], (N)}_{\fP_{\F}}$ over $\cC_{\La}$ then  $\widetilde{R}^{[a,b], (N)} \otimes_{\La} \La'$ represents the extension of $\widetilde{D}^{[a,b], (N)}_{\fP_{\F}}$ restricted to the category $\widehat{\cC}_{\La'}$ and similarly for $ \widetilde{D}^{(N), \mu}_{\fP_{\F}}$ and $\overline{D}^{\mu}_{Q_{\F}}$.    

An argument as in Theorem \ref{universalKisin} shows that an object of $D^{[a,b]}_{\fP_{\F}}(R)$ is the same as a $G$-bundle $\fP_R$ on $\widehat{\fS}_R$ together with a Frobenius $\phi_{\fP_R}: \phz^*(\fP_R)[1/E(u)] \cong \fP_R[1/E(u)]$ deforming $\fP_{\F} \otimes_{\F} R/m_R$ and having height in $[a,b]$. The height in $[a,b]$ condition is essential in order to define the Frobenius over $R$.  We would like to give a criterion for when $(\fP_R, \phi_{\fP_R})$ lies in $D^{\mu}_{\fP_{\F}}(R)$. 

Choose $(\fP_R, \phi_{\fP_R}) \in D^{[a,b]}_{\fP_{\F}}(R)$. For any finite extension $F'$ of $F$ and any homomorphism $x:R \ra F'$, denote the base change of $\fP_R$ to $\fS_{F'}$ by $(\fP_x, \phi_x)$.  Associated to $(\fP_x, \phi_x)$ is a functor $\fD_x$ from $\Rep_F(G_F)$ to filtered $(K \otimes_{\Qp} F')$-modules given by $\fD_x(W) = \phz^*(\fP_x)(W)/E(u) \phz^*(\fP_x)(W)$ with the filtration defined as in Definition \ref{filtrations}.

\begin{lemma} \label{Dx} For any finite extension $F'$ of $F$ and any $x:R \ra F'$, the functor $\fD_x$ is a tensor exact functor.  
\end{lemma}
\begin{proof} Any such $x$ factors through the ring of integers $\La'$ of $F'$ so that $(\fP_x, \phi_x)$ comes from a pair $(\fP_{x_0}, \phi_{x_0})$ over $\fS_{\La'}$.  Let $\widehat{S}_{\La'}$ (resp. $\widehat{S}_{F'}$) to be the $E(u)$-adic completion of $(W \otimes_{\Zp} \La')[u]$ (resp. $(W \otimes_{\Zp} F')[u]$).  By Lemma \ref{switch}, we can think of $(\fP_{x_0}, \phi_{x_0})$ equivalently as a pair over $\widehat{S}_{\La'}$.  

Choose a trivialization $\beta_0$ of $\fP_{x_0}$ and set $Q_{x_0} := \phz^*(\fP_{x_0})$ with trivialization $\delta_{x_0} :=  \beta_0[1/E(u)] \circ \phi_{x_0}$.  Define $(Q_{x}, \delta_x)$ to be $(Q_{x_0}, \delta_{x_0}) \otimes_{\widehat{S}_{\La'}} \widehat{S}_{F'}$ and define a filtration on $\fD_{Q_x} := Q_{x} \mod E(u) $ by 
$$
\Fil^i(\fD_{Q_x}(W)) = (Q_{x}(W) \cap E(u)^i (W \otimes \widehat{S}_{F'}))/ (E(u)Q_{x}(W) \cap E(u)^i (W \otimes \widehat{S}_{F'}))
$$
for any $W \in \Rep_F(G_F)$. Since $\widehat{S}_{\La'}[1/p] /(E(u)) = \widehat{S}_{F'}/(E(u))$, there is an isomorphism
$$
\fD_x \cong \fD_{Q_x}
$$
of tensor exact functors to $\Mod_{K \otimes_{\Qp} F'}$ identifying the filtrations. 

It suffices then to show that $\fD_{Q_x}$ is a tensor exact functor to the category of filtered $(K \otimes_{\Qp} F')$-modules.  Without loss of generality, we assume that $F'$ contains a Galois closure of $K$. Then 
$$
\widehat{S}_{F'} \cong \prod_{\psi} F'[[u - \psi(\pi)]]
$$
over embeddings $\psi:K \ra F'$ (first decompose $W \otimes_{\Zp} F'$ and then decompose $E(u)$ in each factor). Thus, $(Q_x, \delta_x)$ decomposes as a product $\prod_{\psi} (Q^{\psi}_x, \delta^{\psi}_x)$ where each pair defines a point $z_{\psi}$ of the affine Grassmanian of $G_{F'}$.  The quotient $\fD_{Q_x}$ decomposes compatibly as $\prod_{\psi} \fD_{Q^{\psi}_x}$.  We are reduced then to a computation for a point $z_{\psi} \in \Gr_{G_{F'}}(F')$. Without loss of generality, we can assume $G_{F'}$ is split. Up to translation by the positive loop group (which induces an isomorphism on filtrations), $z_{\psi}$ is the image $[g]$ for some $g \in T(F'((t)))$ where $T$ is maximal split torus of $G_{F'}$. Using the weight space decomposition for $T$ on any representation $W$, one can compute directly that $\fD_{Q^{\psi}_x}$ is a tensor exact functor.  For more details, see \cite[Proposition 3.5.11, Lemma 8.2.15]{LevinThesis}.
\end{proof}    

\begin{defn} \label{defnHtype2} Let $F'$ be any finite extension of $F$ with ring of integers $\La'$. We say a $G$-Kisin module $(\fP_{\La'}, \phi_{\La'})$ over $\La'$ has \emph{$p$-adic Hodge type $\mu$} if the $G_F$-filtration associated to $\fP_{\La'} [1/p]$ as above has type $\mu$.
\end{defn}
 
\begin{thm} \label{equivoftype} Assume that $F = F_{[\mu]}$.  Let $R$ be any complete local Noetherian $\La$-algebra with finite residue field which is $\La$-flat and reduced.  Then $\fP_R \in  D^{[a,b]}_{\fP_{\F}}(R)$ lies in $D^{\mu}_{\fP_{\F}}(R)$ if and only if for all finite extensions $F'/F$ and all homomorphisms $x:R \ra F'$, the $G_{F}$-filtration $\fD_x$ has type less than or equal to $[\mu]$. 
\end{thm}  
\begin{proof}
Choose a lift $\widetilde{y}$ of $\fP_R$ to $\widetilde{D}^{[a,b], (N)}_{\fP_{\F}}(R)$. Clearly, $\fP_R \in D^{\mu}_{\fP_{\F}}(R)$ if and only if $\tilde{y} \in \widetilde{D}^{(N), \mu}_{\fP_{\F}}(R)$ which happens if and only $\Psi(\tilde{y}) \in \overline{D}^{\mu}_{Q_{\F}}(R)$.  Let $R^{\mu}$ be the quotient of $R$ representing the fiber product $\Spf R \times_{\overline{D}_{Q_{\F}}^{[a,b]}} \overline{D}^{\mu}_{Q_{\F}}$. To show that $R^{\mu} = R$, it suffices to show that $\Spec R^{\mu}[1/p]$ contains all closed points of $\Spec R[1/p]$ since $R$ is flat and $R[1/p]$ is reduced and Jacobson.

The groupoid $\overline{D}^{\mu}_{Q_{\F}}$ is represented by a completed stalk on the local model $M(\mu) \subset \Gr_{G}^{E(u), W}$ so that for any $x:R \ra F'$, $\Psi(\widetilde{y})[1/p]$ defines a $F'$-point $(Q_x, \delta_x)$ of $\Gr_{G}^{E(u), W}$.   Since $M(\mu)(F') = S(\mu)(F')$, $(Q_x, \delta_x) \in S(\mu)(F')$ if and only if the filtration $\fD_{Q_x}$ has type $\leq [\mu]$ (\cite[Proposition 3.5.11]{LevinThesis}). The proof of Lemma \ref{Dx} shows that the two filtrations agree, i.e.,
$$
\fD_x \cong \fD_{Q_x}.
$$  
Thus, $x$ factors through $R^{\mu}$ exactly when the type of the filtration $\fD_x$ is less than or equal to $[\mu]$. 
\end{proof} 

Fix a continuous representation $\overline{\eta}:\Gamma_K \ra G(\F)$. Let $R^{[a,b], \cris}_{\overline{\eta}}$ be the universal framed $G$-valued crystalline deformation ring with Hodge-Tate weights in $[a,b]$, and let $\Theta:X^{[a,b], \cris}_{\overline{\eta}} \ra \Spec R^{[a,b], \cris}_{\overline{\eta}}$ be as in \ref{resolution}.   

\begin{defn} \label{leqmu} Assume $F = F_{[\mu]}$. Define $R^{\cris, \leq \mu}_{\overline{\eta}}$ to be the flat closure of the connected components of $\Spec R^{[a,b], \cris}_{\overline{\eta}}[1/p]$ with type $\leq \mu$ (see \ref{ssmutype}). Define $X^{\cris, \leq \mu}_{\overline{\eta}}$ to be the flat closure in $X^{[a,b], \cris}_{\overline{\eta}}$ of the same connected components (since $\Theta[1/p]$ is an isomorphism).   
\end{defn}

\begin{cor} \label{cristypes} Let $X^{\cris, \leq \mu}_{\overline{\eta}}$ be as in Definition $\ref{leqmu}$. A point $\overline{x} \in X^{\cris, \leq \mu}_{\overline{\eta}}(\F')$ corresponds to a $G$-Kisin lattice $\fP_{\F'}$ over $\fS_{\F'}$. The deformation problem $D^{\cris, \mu}_{\overline{x}}$ which assigns to any $A \in \cC_{\La \otimes_{\Zp} W(\F')}$ the set of  isomorphisms classes of triples $$\{(y:R^{\cris, \leq \mu}_{\overline{\eta}} \ra A, \fP_A \in D^{\mu}_{\fP_{\F'}}(A), \delta_A:T_{G, \fS_A}(\fP_A) \cong \eta_y|_{\Gamma_{\infty}})\}$$ is representable. Furthermore, if $\widehat{\cO}^{\mu}_{\overline{x}}$ is the completed local ring of $X^{\cris, \leq \mu}_{\overline{\eta}}$ at $\overline{x}$, then the natural map $\Spf \widehat{\cO}^{\mu}_{\overline{x}} \ra D^{\cris, \mu}_{\overline{x}}$ is a closed immersion which is an isomorphism modulo $p$-power torsion.  
\end{cor} 
\begin{proof}
Without loss of generality, we can replace $\La$ by $\La \otimes_{W(\F)} W(\F')$. By construction and Proposition \ref{fhfunctoriality}, for any $A \in \cC_{\La}$, the deformation functor 
$$
D^{\cris, \mu, \mathrm{bc}}_{\overline{x}}(A) = \{y:R^{\cris, \leq \mu}_{\overline{\eta}} \ra A, \fP_A \in D^{[a,b]}_{\fP_{\F'}}(A), \delta_A:T_{G, \fS_A}(\fP_A) \cong \eta_y|_{\Gamma_{\infty}}\}/\cong
$$
is representable.  That is, $D^{\cris, \mu,  \mathrm{bc}}_{\overline{x}}$ represents the completed stalk at a point of the fiber product $X^{[a,b], \cris}_{\overline{\eta}} \times_{\Spec R^{[a,b], \cris} _{\overline{\eta}}} \Spec R^{\cris, \leq \mu}_{\overline{\eta}}$. Since $D^{\mu}_{\fP_{\F'}} \subset D^{[a,b]}_{\fP_{\F'}}$ is closed so is $D^{\cris, \mu}_{\overline{x}} \subset D^{\cris, \mu,  \mathrm{bc}}_{\overline{x}}$ and hence $D^{\cris, \mu}_{\overline{x}}$ is representable by $R^{\cris, \mu}_{\overline{x}}$. To see that the closed immersion $\Spf \widehat{\cO}^{\mu}_{\overline{x}} \ra D^{\cris, \mu,  \mathrm{bc}}_{\overline{x}}$ factors through $D^{\cris, \mu}_{\overline{x}}$ it suffices to show that the ``universal" lattice $\fP_{\widehat{\cO}^{\mu}_{\overline{x}}} \in D^{[a,b]}_{\fP_{\F'}}(\widehat{\cO}^{\mu}_{\overline{x}})$ lies in $D^{\mu}_{\fP_{\F'}}(\widehat{\cO}^{\mu}_{\overline{x}})$.

By Theorem \ref{fhimage} and \ref{stlocus}, $\Theta[1/p]$ is an isomorphism.  Furthermore, by \cite[Proposition 5.1.5]{Balaji}, $R^{[a,b], \cris}_{\overline{\eta}}[1/p]$ and $R^{\cris, \leq \mu}_{\overline{\eta}}[1/p]$ are formally smooth over $F$.   Hence,  $\widehat{\cO}^{\mu}_{\overline{x}}$ satisfies the hypotheses of Theorem \ref{equivoftype}.  

By Theorem \ref{equivoftype}, we are reducing to showing that for any finite $F'/F$ and any homomorphism $x:\widehat{\cO}^{\mu}_{\overline{x}} \ra F'$ the filtration $\fD_x$ corresponding to the base change $\fP_x := \fP_{\widehat{\cO}^{\mu}_{\overline{x}}} \otimes_x F'$ has type less than or equal to $\mu$.  The homomorphism $x$ corresponds to closed point of $\Spec R^{\cris, \leq \mu}_{\overline{\eta}}[1/p]$, i.e., a crystalline representation $\rho_x$ with $p$-adic Hodge type $\leq \mu$. Furthermore, $\fP_x$ is the unique $(\fS_{F'}, \phz)$-module of bounded height associated to $\rho_x$. By Proposition \ref{DRfiltcomparison}, the de Rham $\cF^{\mathrm{dR}}_{\rho_x}$ filtration associated to $\rho_x$ is isomorphic to the filtration $\fD_x$ associated to $(\fP_x, \phi_x)$. Thus $\fD_x$ has type $\leq \mu$ for all points $x$ and so $\fP_{\widehat{\cO}^{\mu}_{\overline{x}}} \in D^{\mu}_{\fP_{\F'}}(\widehat{\cO}^{\mu}_{\overline{x}})$ by Theorem \ref{equivoftype}.

By the argument above, $\Spec \widehat{\cO}^{\mu}_{\overline{x}}$ and $\Spec R^{\cris, \mu}_{\overline{x}}$ have the same $F'$-points for any finite extension of $F$.  Since $R^{\cris, \leq \mu}_{\overline{\eta}}[1/p]$ is formally smooth over $F$, the kernel of $R^{\cris, \mu}_{\overline{x}} \ra \widehat{\cO}^{\mu}_{\overline{x}}$ is $p$-power torsion.  
\end{proof}

\begin{rmk} In fact, Corollary \ref{cristypes} holds as well for semistable deformation rings with $p$-adic Hodge type $\leq \mu$.  To apply Theorem \ref{equivoftype} and make the final deduction, we  needed that the generic fiber of the crystalline deformation ring was reduced (to argue at closed points).  This is true for $G$-valued semistable deformation rings by the main result of \cite{Bellovin2}.  
\end{rmk}

\section{Local analysis}

In this section, we analyze finer properties of crystalline $G$-valued deformation rings with minuscule $p$-adic Hodge type.  The techniques in this section are inspired by \cite{MFFGS} and \cite{Liu2}.  We develop a theory of $(\phz, \widehat{\Gamma})$-modules with $G$-structure and our main result, Theorem \ref{keythm}, is stated in these terms.  However, the idea is the following: given a $G$-Kisin module $(\fP_A, \phi_A)$ over some finite $\La$-algebra $A$, we get a representation of $\Gamma_{\infty}$ via the functor $T_{G, \fS_A}$.  In general, this representation need not extend (and certainly not in a canonical way) to a representation of the full Galois group $\Gamma_K$.  When $G = \GL_n$ and $\fP_A$ has height in $[0, 1]$ then via the equivalence between Kisin modules with height in $[0,1]$ and finite flat group schemes \cite[Theorem 2.3.5]{Fcrystals}, one has a canonical extension to $\Gamma_K$ which is flat.  We show (at least when $A$ is a $\La$-flat domain) that the same holds for $G$-Kisin modules of minuscule type: there exists a canonical extension to $\Gamma_K$ which is crystalline. This is stated precisely in Corollary \ref{BTtype}.  We end by applying this result to identify the connected components of $G$-valued crystalline deformation rings with the connected components of a moduli space of $G$-Kisin modules (Corollary \ref{connectedcomponents}).

\subsection{Minuscule cocharacters}

We begin with some preliminaries on minuscule cocharacters and adjoint representations which we use in our finer analysis with $(\phz, \widehat{\Gamma})$-modules in the subsequent sections.   

Let $H$ be a reductive group over field $\kappa$.  The conjugation action of $H$ on itself gives a representation
\begin{equation} \label{AdRep}
\Ad:H \ra \GL(\Lie(H)).
\end{equation}
This is an algebraic representation so for any $\kappa$-algebra $R$, $H(R)$ acts on $\Lie(H_R) = \Lie H \otimes_{\kappa} R$.  We will use $\Ad$ to denote these actions as well.  We can define $\Ad$ for $G$ over $\Spec \La$ in the same way. 

\begin{defn} \label{liegrading} Any cocharacter $\la:\Gm \ra H$ gives a grading on $\Lie H$ defined by 
$$
\Lie H (i) := \{ Y \in \Lie H \mid \Ad(\la(a)) Y = a^i Y \}.
$$   
A cocharacter $\la$ is called \emph{minuscule} if $\Lie H(i) = 0$ for $i \not\in \{ -1, 0,  1 \}$. 
\end{defn}

Minuscule cocharacters were studied by Deligne \cite{Deligne} in connection with the theory of Shimura varieties.  A detailed exposition of their main properties can be found \S 1 of \cite{Gross}. 


Assume now that $H$ is split and fix a maximal split torus $T$ contained in Borel subgroup $B$.  This gives rise to a set of simple roots $\Delta$ and a set of simple coroots $\Delta^{\vee}$.  In each conjugacy class of cocharacters, there is a unique dominant cocharacter valued in $T$.  The set of dominant cocharacters is denoted by $X_*(T)^{+}$.  

Recall the Bruhat (partial) ordering on $X_*(T)^{+}$: given two dominant cocharacter $\mu, \mu':\Gm \ra T$, we say $\mu' \leq \mu$ if $\mu - \mu' = \sum_{\alpha \in \Delta^{\vee}} n_{\alpha} \alpha$  with $n_{\alpha} \geq 0$.  

\begin{prop} \label{smallest} Let $\mu$ be a dominant minuscule cocharacter.  Then there are no dominant $\mu'$ such that $\mu' < \mu$ in the Bruhat order. 
\end{prop}
\begin{proof} See Exercise 24 from Chapter IV.1 of \cite{BourbakiLie}.    
\end{proof}

\begin{prop} \label{closedschubert} If $\mu$ is a minuscule cocharacter, then the $($open$)$ affine Schubert variety $S^0(\mu)$ is equal to $S(\mu)$.  Furthermore, $S(\mu)$ is smooth and projective.  In fact, $S(\mu) \cong H/P(\mu)$ where $P(\mu)$ is a parabolic subgroup associated to the cocharacter $\mu$.  
\end{prop} 
\begin{proof} Since the closure $S(\mu) = \cup_{\mu' \leq \mu} S^0(\mu')$  (\cite[Proposition 2.8]{RicharzSV}), the first part follows from Proposition \ref{smallest} .  For the remaining facts, we refer to discussion after \cite[Definition 1.3.5]{PRS} and \cite[Proposition 3.5.7]{LevinThesis}. 
\end{proof}

For any $\mu:\Gm \ra T$, we get an induced map $\Gm(\kappa(\!(t)\!)) \ra T(\kappa(\!(t)\!)) \subset H(\kappa(\!(t)\!))$ on loop groups.   We let $\mu(t)$ denote the image of $t \in \kappa(\!(t)\!)^{\times}$.  

\begin{prop} \label{Adpole} For any $X \in \Lie H \otimes_{\kappa} \kappa[\![t]\!]$,  we have 
$$
\Ad(\mu(t))(X) \in \frac{1}{t}(\Lie H \otimes_{\kappa} \kappa[\![t]\!]).
$$
\end{prop} 
\begin{proof} As in Definition \ref{liegrading}, we can decompose $\Lie H = \Lie H(-1) \oplus \Lie H \oplus \Lie H(1)$. Then $\Ad(\mu(t))$ acts on $\Lie H(i) \otimes \kappa(\!(t)\!)$ by multiplication by $t^{i}$.  The largest denominator is then $t^{-1}$.  
\end{proof}

\subsection{$(\phz, \widehat{\Gamma})$-modules with $G$-structure}
We review Liu's theory of $(\phz, \widehat{G})$ as in \cite{Liu1, CL}.  We will call them $(\phz, \widehat{\Gamma})$-modules to avoid confusion with the algebraic group $G$. The theory of  $(\phz, \widehat{\Gamma})$-modules is an adaptation of the theory of $(\phz, \Gamma)$-modules to the non-Galois extension $K_{\infty} = K(\pi^{1/p}, \pi^{1/p^2}, \ldots)$. The $\widehat{\Gamma}$ refers to an additional structure added to a Kisin module which captures the full action of $\Gamma_K$ as opposed to just the subgroup $\Gamma_{\infty} := \Gal(\overline{K}/K_{\infty})$. The main theorem in \cite{Liu1} is an equivalence of categories between (torsion-free) $(\phz, \widehat{\Gamma})$-modules and $\Gamma_K$-stable lattices in semi-stable $\Qp$-representations.      

Let $\widetilde{\bE}^+$ denote the perfection of $\cO_{\overline{K}}/(p)$. There is a unique surjective map
$$
\Theta:W(\widetilde{\bE}^+) \ra \widehat{\cO}_{\overline{K}}
$$
which lifts the projection $\widetilde{\bE}^+ \ra \cO_{\overline{K}}/(p)$. The compatible system $(\pi^{1/p^n})_{n \geq 0}$ of the $p^n$th roots of $\pi$ defines an element $\underline{\pi}$ of $\widetilde{\bE}^+$. Let $[\underline{\pi}]$ denote the Teichm\"uller representative in $W(\widetilde{\bE}^+)$.  There is an embedding 
$$
\fS \ia W(\widetilde{\bE}^+)
$$
defined by $u \mapsto [\underline{\pi}]$ which is compatible with the Frobenii. If $\widetilde{\bE}$ is the fraction field of $\widetilde{\bE}^+$, then $W(\widetilde{\bE}^+) \subset W(\widetilde{\bE})$. The embedding $\fS \iarrow W(\widetilde{\bE}^+)$ extends to an embedding
$$
\cO_{\cE} \iarrow W(\widetilde{\bE}).
$$

As before, let $K_{\infty} = \bigcup K(\pi^{1/p^n})$.  Set $K_{p^{\infty}} := \bigcup K(\zeta_{p^n})$ where $\zeta_{p^n}$ is a primitive $p^n$th root of unity. Denote the compositum of $K_{\infty}$ and $K_{p^{\infty}}$ by $K_{\infty, p^{\infty}}$; $K_{\infty, p^{\infty}}$ is Galois over $K$.

\begin{defn} Define $$\widehat{\Gamma} := \Gal(K_{\infty, p^{\infty}}/K) \text{ and } \widehat{\Gamma}_{\infty} := \Gal(K_{\infty, p^{\infty}}/K_{\infty}).$$
\end{defn}

There is a subring $\widehat{R} \subset W(\widetilde{\bE}^+)$ which plays a central role in the theory of $(\phz, \widehat{\Gamma})$-modules. The definition can be found on page 5 of \cite{Liu1}. The relevant properties of $\widehat{R}$ are:
\begin{enumerate} [$(1)$]
\item $\widehat{R}$ is stable by the Frobenius on $W(\widetilde{\bE}^+)$;
\item $\widehat{R}$ contains $\fS$;
\item $\widehat{R}$ is stable under the action of the Galois group $\Gamma_K$ and $\Gamma_K$ acts through the quotient $\widehat{\Gamma}$. 
\end{enumerate}

For any $\Zp$-algebra $A$, set $\widehat{R}_A := \widehat{R} \otimes_{\Zp} A$ with a Frobenius induced by the Frobenius on $\widehat{R}$.  Similarly, define $W(\widetilde{\bE}^+)_A := W(\widetilde{\bE}^+) \otimes_{\Zp}  A$ and $W(\widetilde{\bE})_A := W(\widetilde{\bE}) \otimes_{\Zp}  A$.  For any $\fS_A$-module $\fM_A$, define
$$
\widehat{\fM}_A := \widehat{R}_A \otimes_{\phz, \fS_A} \fM_A = \widehat{R}_A \otimes_{\fS_A} \phz^*(\fM_A)
$$
and
$$
\widetilde{\fM}_A := W(\widetilde{\bE}^+)_A \otimes_{\phz, \fS_A} \fM_A = W(\widetilde{\bE}^+)_A \otimes_{\widehat{R}_A} \widehat{\fM}_A.
$$
Recall that $\phz^*(\fM_A) := \fS_A \otimes_{\phz, \fS_A} \fM_A$ and that the linearized Frobenius is a map $\phi_{\fM_A}:\phz^*(\fM_A) \ra \fM_A$ (when $\fM_A$ has height in $[0, \infty)$).

If $\fM_A$ is a projective $\fS_A$-module then, by Lemma 3.1.1 in \cite{CL}, $\phz^*(\fM_A) \subset \widehat{\fM}_A \subset \widetilde{\fM}_A$.   Although the map $m \mapsto 1 \otimes m$ from $\fM_A$ to $\widehat{\fM}_A$ is not $\fS_A$-linear, it is injective when $\fM_A$ is $\fS_A$-projective.  The image is a $\phz(\fS_A)$-submodule of $\widehat{\fM}_A$.  We will think of $\fM_A$ inside of $\widehat{\fM}_A$ in this way.   Finally, for any \'etale $(\cO_{\cE, A}, \phz)$-module $\cM_A$, we define 
$$
\widetilde{\cM}_A := W(\widetilde{\bE})_A \otimes_{\phz, \cO_{\cE, A}} \cM_A = W(\widetilde{\bE})_A \otimes_{\cO_{\cE, A}} \phz^*(\cM_A) 
$$
with semi-linear Frobenius extending the Frobenius on $\cM_A$.  To summarize,  for any Kisin module $(\fM_A, \phi_A)$ , we have the following diagram
$$
\xymatrix{  
(\fM_A, \phi_A) \ar@{~>}[d] \ar@{~>}[r] & \widehat{\fM}_A \ar@{~>}[r] & \widetilde{\fM}_A \ar@{~>}[d] \\
(\cM_A, \phi_A) \ar@{~>}[rr] & & (\widetilde{\cM}_A, \widetilde{\phi}_A). \\ 
}
$$ 

Now, let $\gamma \in \widehat{\Gamma}$ and let $\widehat{\fM}_A$ be an $\widehat{R}_A$-module.  A map $g:\widehat{\fM}_A \ra \widehat{\fM}_A$ is a \emph{$\gamma$-semilinear} if
$$
g(a m) = \gamma(a) g(m)
$$
for any $a \in \widehat{R}_A, m \in \widehat{\fM}_A$. A (semilinear) $\widehat{\Gamma}$-action on $\widehat{\fM}_A$ is a $\gamma$-semilinear map $g_{\gamma}$ for each $\gamma \in \widehat{\Gamma}$ such that 
$$
g_{\gamma'} \circ g_{\gamma} = g_{\gamma' \gamma}
$$
as $(\gamma' \gamma)$-semilinear morphisms. A (semilinear) $\widehat{\Gamma}$-action on $\widehat{\fM}_A$ extends in the natural way to a (semilinear) $\Gamma_K$-action on $\widetilde{\fM}_A$ and on $\widetilde{\cM}_A$.

For any local Artinian $\Zp$-algebra $A$, choose a $\Zp$-module isomorphism $A \cong \bigoplus \Z/p^{n_i} \Z$ so that as a $W(\widetilde{\bE})$-module, $W(\widetilde{\bE})_A \cong \bigoplus W_{n_i} (\widetilde{\bE})$.  We equip $W(\widetilde{\bE})_A$ with the product topology where $W_{n_i}(\widetilde{\bE})$ has a topology induced by the isomorphism  $W_{n_i}(\widetilde{\bE}) \cong \widetilde{\bE}^{n_i}$ given by Witt components (see \S 4.3 of \cite{PHT} for more details on the topology of $\widetilde{\bE}$).  We can similarly define a topology on $W(\widetilde{\bE}^+)_A$ using the topology on $\widetilde{\bE}^+$, and it is clear that  this is the same as the subspace topology from the inclusion $W(\widetilde{\bE}^+)_A \subset W(\widetilde{\bE})_A$. Finally, we give $\widehat{R}_A$ the subspace topology from the inclusion $\widehat{R}_A \subset W(\widetilde{\bE}^+)_A$. The same procedure works for $A$ finite flat over $\Zp$. 


A $\widehat{\Gamma}$-action on $\widehat{\fM}_A$ is \emph{continuous} if for any basis (equivalently for all bases) of $\widehat{\fM}_A$ the induced map $\widehat{\Gamma} \ra \GL_r(\widehat{R}_A)$ is continuous where $r$ is the rank of $\widehat{\fM}_A$ (such a basis exists by \cite[Lemma 1.2.2(4)]{MFFGS}). 

\begin{defn} \label{hatmodule} Let $A$ be a finite $\Zp$-algebra. A \emph{$(\phz, \widehat{\Gamma})$-module with height in $[a,b]$} over $A$ is a triple $(\fM_A, \phi_{\fM_A}, \widehat{\Gamma})$, where
\begin{enumerate} [$(1)$]
\item $(\fM_A, \phi_{\fM_A}) \in \Mod^{\phz, [a,b]}_{\fS_A}$;
\item $\widehat{\Gamma}$ is a continuous (semilinear) $\widehat{\Gamma}$-action on $\widehat{\fM}_A$;
\item The $\Gamma_K$-action on $\widetilde{\fM}_A$ commutes with $\widetilde{\phi}_{\fM_A}$ (as endomorphisms of $\widetilde{\cM}_A$);
\item Regarding $\fM_A$ as a $\phz(\fS_A)$-submodule of $\widehat{\fM}_A$, we have $\fM_A \subset \widehat{\fM}^{\widehat{\Gamma}_{\infty}}_A$;
\item $\widehat{\Gamma}$ acts trivially on $\widehat{\fM}_A/I_+(\widehat{\fM}_A)$ (see \S 3.1 of \cite{CL} for the definition of $I_+(\widehat{\fM}_A)$).  
\end{enumerate}
We often refer to the additional data of a $(\phz, \widehat{\Gamma})$-module on a Kisin module as a \emph{ $\widehat{\Gamma}$-structure}. 
\end{defn}

\begin{rmk} Although we allow arbitrary height $[a,b]$ (in particular, negative height), the ring $\widehat{R}$ is still sufficient for defining the $\widehat{\Gamma}$-action.  This follows from the fact that the $\widehat{\Gamma}$-action on $\fS(1)$ is given by $\widehat{c}$ (see \cite[Example 3.2.3]{Liu1}) which is a unit in $\widehat{R}$.  See also \cite[Example 9.1.9]{LevinThesis}.  
\end{rmk} 
 
\begin{prop} \label{commconcrete} Choose $(\fM_A, \phi_{\fM_A}) \in \Mod^{\phz, [a,b]}_{\fS_A}$ of rank $r$. Fix a basis $\{ f_i \}$ for $\fM_A$.  Let $C'$ be the matrix for $\phi_{\fM_A}$ with respect to $\{1 \otimes_{\phz} f_i\}$. Then a $\widehat{\Gamma}$-structure on $\fM_A$ is the same as a continuous map 
$$
B_{\bullet}:\widehat{\Gamma} \ra \GL_r(\widehat{R}_A)
$$
such that 
\begin{enumerate} [$(a)$]
\item  $C' \cdot \phz(B_{\gamma}) = B_{\gamma} \cdot \gamma(C') $ in $\Mat(W(\widetilde{\bE})_A)$ for all $\gamma \in \widehat{\Gamma}$;
\item $B_{\gamma} = \Id$ for all $\gamma \in \widehat{\Gamma}_{\infty}$;
\item $B_{\gamma} \equiv \Id \mod I_+(\widehat R)_A$ for all $\gamma \in \widehat{\Gamma}$;
\item $B_{\gamma \gamma'} = B_{\gamma} \cdot \gamma(B_{\gamma'})$ for all $\gamma, \gamma' \in \widehat{\Gamma}$.
\end{enumerate}
\end{prop} 

Let $\Mod^{\phz, [a,b], \widehat{\Gamma}}_{\fS_A}$ denote the category of $(\phz, \widehat{\Gamma})$-modules with height in $[a,b]$ over $A$. A morphism between $(\phz, \widehat{\Gamma})$-modules is a morphism in $\Mod^{\phz, [a,b]}_{\fS_A}$ that is $\widehat{\Gamma}$-equivariant when extended to $\widehat{R}_A$.

Let $\Mod^{\phz, \bh, \widehat{\Gamma}}_{\fS_A} := \bigcup_{h > 0} \Mod^{\phz, [-h,h], \widehat{\Gamma}}_{\fS_A}$ so $\Mod^{\phz, \bh, \widehat{\Gamma}}_{\fS_A}$ has a natural tensor product operation which at the level of $\Mod^{\phz, \bh}_{\fS_A}$ is tensor product of bounded height Kisin modules.  The $\widehat{\Gamma}$-structure on the tensor product is defined via
\begin{equation*} \label{tensorhatequat}
\widehat{R}_A \otimes_{\phz, \fS_A} (\fM_A \otimes_{\fS_A} \fN_A) \cong (\widehat{R}_A \otimes_{\phz, \fS_A} \fM_A) \otimes_{\widehat{R}_A} (\widehat{R}_A \otimes_{\phz, \fS_A} \fN_A) = \widehat{\fM}_A \otimes_{\widehat{R}_A} \widehat{\fN}_A.
\end{equation*} 
One also defines a $\widehat{\Gamma}$-structure on the dual $\fM^{\vee}_A := \Hom_{\fS_A}(\fM_A, \fS_A)$ in the natural way (see discussion after Proposition 9.1.5 \cite{LevinThesis}). Note that, unlike in other references (for example \cite{Ozeki}), we do not include any Tate twist in our definition of duals.  

We will now relate these $(\phz, \widehat{\Gamma})$-modules to $\Gamma_K$-representations. For this, we require that $A$ be $\Zp$-finite and either $\Zp$-flat or Artinian.   Define a functor $\widehat{T}_A$ from $\Mod^{\phz, \bh, \widehat{\Gamma}}_{\fS_A}$ to Galois representations by
$$
\widehat{T}_A(\widehat{\fM}_A) := (W(\widetilde{\bE}) \otimes_{\widehat{R}} \widehat{\fM}_A)^{\widetilde{\phi}_A = 1} = (\widetilde{\cM}_A)^{\widetilde{\phi}_A = 1}. 
$$
The semilinear $\Gamma_K$-action on $\widetilde{\cM}_A$ commutes with $\widetilde{\phi}_A$ so $\widehat{T}_A(\widehat{\fM}_A)$ is a $\Gamma_K$-stable $A$-submodule of $W(\widetilde{\bE}) \otimes_{\widehat{R}} \widehat{\fM}_A$. 

We now recall the basic facts we will need about $\widehat{T}_A$:
\begin{prop} \label{hatT} Let $A$ be $\Zp$-finite and either $\Zp$-flat or Artinian.  
\begin{enumerate}
\item If $\widehat{\fM}_A \in \Mod^{\phz, \bh, \widehat{\Gamma}}_{\fS_A}$, then there is a natural $A[\Gamma_{\infty}]$-module isomorphism 
\begin{equation*}
\theta_A:T_{\fS_A}(\fM_A) \ra \widehat{T}_A(\widehat{\fM}_A).
\end{equation*}
Furthermore, $\theta_A$ is functorial with respect to morphisms in $\Mod^{\phz, \bh, \widehat{\Gamma}}_{\fS_A}$. 
\item  $\widehat{T}_A$ is an exact tensor functor from $\Mod^{\phz, \bh, \widehat{\Gamma}}_{\fS_A}$ to $\Rep_A(\Gamma_K)$ which is compatible with duals. 
\end{enumerate}
\end{prop}
\begin{proof} See Propositions 9.1.6 and 9.1.7 \cite{LevinThesis}. 
\end{proof} 

We are now ready to add $G$-structure to $(\phz, \widehat{\Gamma})$-modules.  Let $G$ be a connected reductive group over a $\Zp$-finite and flat local domain $\La$ as in previous sections. 

\begin{defn} Define $\GMod^{\phz, \widehat{\Gamma}}_{\fS_A}$ to be the category of faithful exact tensor functors $[\sideset{^f}{_{\La}}\Rep (G), \Mod^{\phz, \bh, \widehat{\Gamma}}_{\fS_A}]^{\otimes}$. We will refer to these as \emph{$(\phz, \widehat{\Gamma})$-modules with $G$-structure}.
\end{defn}



Recall the category $\GRep_A(\Gamma_K)$ from Definition \ref{GREP}.  By Proposition \ref{hatT}(2), $\widehat{T}_A$ induces a functor
$$
\widehat{T}_{G, A}:\GMod^{\phz, \widehat{\Gamma}}_{\fS_A} \ra \GRep_A(\Gamma_K).
$$
Furthermore, if $\omega_{\Gamma_{\infty}}:\GRep_A(\Gamma_K) \ra \GRep_A(\Gamma_{\infty})$ is the forgetful functor then there is an natural isomorphism
$$
T_{G, \fS_A} \cong \omega_{\Gamma_{\infty}} \circ \widehat{T}_{G, A}.
$$
The functor $\widehat{T}_{G, A}$ behaves well with respect to base change along finite maps $A \ra A'$ by the same argument as in Proposition \ref{normGequiv}.  

We end this section by adding $G$-structure to the main result of \cite{Liu1}.  For $A$ finite flat over $\La$, an element $(P_A, \rho_A)$ of $\GRep_A(\Gamma_{K})$ is \emph{semi-stable} (resp. \emph{crystalline}) if $\rho_A[1/p]:\Gamma_K \ra \Aut_G(P_A) (A[1/p])$ is semi-stable (resp. crystalline). For $A$ a local domain, and $\rho_A$ semi-stable, we say $\rho_A$ has \emph{$p$-adic Hodge type $\mu$} if $\rho_A[1/p]$ does for any trivialization of $P_A$ (see Definition \ref{defnHtype}). 
  
\begin{thm} \label{GLiu} Let $F'$ be a finite extension of $F$ with ring of integers $\La'$. The functor $\widehat{T}_{G, \La'}$ induces an equivalence of categories between $\GMod^{\phz, \widehat{\Gamma}}_{\fS_{\La'}}$ and the full subcategory of semi-stable representations of $\GRep_{\La'}(\Gamma_K)$.
\end{thm}
\begin{proof} Using the Tannakian description of both categories, it suffices to show that $\widehat{T}_{\La'}$ defines a tensor equivalence between $\Mod^{\phz, \bh, \widehat{\Gamma}}_{\fS_{\La'}}$ and semi-stable representations of $\Gamma_K$  on finite free $\La'$-modules.  When $F = \Qp$ and the Hodge-Tate weights are negative (in our convention), this is Theorem 2.3.1 in \cite{Liu1}.  Note that \cite{Liu1} is using contravariant functors so that our $\widehat{T}_{\La'}$ is obtained by taking duals. The restriction on Hodge-Tate weights can be removed by twisting by $\widehat{\fS}(1)$,  the $(\phz, \widehat{\Gamma})$-module corresponding to the inverse of the $p$-adic cyclotomic character.  

To define a quasi-inverse to $\widehat{T}_{\La'}$, let $L$ be a semi-stable $\Gamma_K$-representation  on a finite free $\La'$-module. Forgetting the coefficients, \cite{Liu1} constructs a $\widehat{\Gamma}$-structure $\widehat{T}^{-1}(L)$  on the unique Kisin lattice in $\underline{M}(L)$.   This $(\phz, \widehat{\Gamma})$-module over $\Zp$ has an action of $\La'$ by functoriality of the construction.  By an argument as in \cite[Proposition 1.6.4(2)]{PST}, the resulting $\fS_{\La'}$-module is projective and so this defines an object of $\Mod^{\phz, \bh, \widehat{\Gamma}}_{\fS_{\La'}}$ which we call $\widehat{T}_{\La'}^{-1}(L)$.

Finally, we appeal to Proposition I.4.4.2 in \cite{Saavedra} to conclude that $\widehat{T}_{\La'}$ and $\widehat{T}_{\La'}^{-1}$ define a tensor equivalence of categories given that $\widehat{T}_{\La'}$ respects tensor products (Proposition \ref{hatT}).
\end{proof}

\subsection{Faithfulness and existence result}

Fix an element $\tau$ in $\widehat{\Gamma}$ such that $\tau(\underline{\pi}) = \underline{\eps} \cdot \underline{\pi}$ where $\underline{\eps}$ is a compatible system of primitive $p^n$th roots of unity. If $p \neq 2$, then $\tau$ is a topological generator for $\widehat{\Gamma}_{p^{\infty}} := \Gal(K_{\infty, p^{\infty}}/K_{p^{\infty}})$. If $p = 2$, then some power of $\tau$ will generate $\widehat{\Gamma}_{p^{\infty}}$.  In both cases, $\tau$ together with $\widehat{\Gamma}_{\infty}$ topologically generate $\widehat{\Gamma}$ (see \cite[\S 4.1]{Liu1}). Given condition (4) in Definition \ref{hatmodule} the $\widehat{\Gamma}$-action is determined by the action of $\tau$.  

Recall the element $\mathfrak{t} \in W(\widetilde{\bE}^+)$ which is the period for $\fS(1)$ in the sense that $\phz(\mathfrak{t}) = c_0^{-1} E(u) \mathfrak{t}$. We will need a few structural results about $W(\widetilde{\bE}^+)$.

\begin{lemma} \label{tdiv} For any $\widetilde{\gamma} \in \Gamma_K$, we have the following divisibilities in $W(\widetilde{\bE}^+):$ 
$$
\widetilde{\gamma}(u) \mid u, \quad \widetilde{\gamma}(\phz(\ft)) \mid \phz(\ft), \text{ and } \quad \widetilde{\gamma}(E(u)) \mid E(u).
$$
\end{lemma}
\begin{proof} See \cite[Lemma 9.3.1]{LevinThesis}.
\end{proof}

The $(\phz, \widehat{\Gamma})$-modules which give rise to crystalline representations satisfy an extra divisibility condition on the action of $\tau$, i.e., \cite[Cor. 4.10]{BDJ} and \cite[Prop. 9.3.4]{LevinThesis}.  We call this the \emph{crystalline condition}.

\begin{defn} An object $\widehat{\fM}_A \in \Mod^{\phz, [a,b], \widehat{\Gamma}}_{\fS_A}$ is \emph{crystalline} if for any $x \in \fM_A$ there exists $y \in \widetilde{\fM}_A$ such that $\tau(x) - x = \phz(\mathfrak{t}) u^p y$.
\end{defn}

\begin{prop} \label{allgamma} If $\widehat{\fM}_A$ is crystalline, then for all $x \in \fM_A$ and $\gamma \in \widehat{\Gamma}$ there exists $y \in \widetilde{\fM}_A$ such that $\gamma(x) - x = \phz(\mathfrak{t}) u^p y$.
\end{prop}
\begin{proof} This is an easy calculation using that $\widehat{\Gamma}$ is topologically generated by $\widehat{\Gamma}_{\infty}$ and $\tau$ (\cite[Proposition 9.3.3]{LevinThesis}).
\end{proof} 

\begin{defn} We say an object $\widehat{\fP}_{A} \in  \GMod^{\phz, [a,b], \widehat{\Gamma}}_{\fS_A}$ is \emph{crystalline} if $\widehat{\fP}_{A}(W)$ is crystalline for all $W \in \sideset{^f}{_{\La}}\Rep (G)$.  For $\widehat{\fP}_{\F} \in  \GMod^{\phz, [a,b], \widehat{\Gamma}}_{\fS_{\F}}$, define the \emph{crystalline $(\phz, \widehat{\Gamma})$-module deformation groupoid} over $\cC_{\La}$ by 
$$
D^{\text{cris}, [a,b]}_{\widehat{\fP}_{\F}}(A) = \{ (\widehat{\fP}_A, \psi_0) \in D^{[a,b]}_{\widehat{\fP}_{\F}} (A) \mid \widehat{\fP}_A \text{  is crystalline} \}
$$
for any $A \in \cC_{\La}$. 
\end{defn} 

\begin{prop} \label{criscris} Let $F'$ be a finite extension of $F$ with ring of integers $\La'$.  The equivalence from Theorem \ref{GLiu} induces an equivalence between the full subcategory of crystalline objects in $\GMod^{\phz, \widehat{\Gamma}}_{\fS_{\La'}}$ with the category of crystalline representations in $\GRep_{\La'}(\Gamma_K)$.
\end{prop}
\begin{proof} It suffices to show that if $\widehat{T}_A(\widehat{\fP}_{A}(W))$ is a lattice in a crystalline representation then $\widehat{\fP}_{A}(W)$ satisfies the crystalline condition.  This only depends on the underlying $(\phz, \widehat{\Gamma})$-module so we can take $A = \Zp$.  When $p > 2$, this is proven in Corollary 4.10 in \cite{BDJ}.  The argument for $p = 2$ is essentially the same and was omitted only because in \cite{BDJ} they need further divisibilities on $(\tau - 1)^n$ for which $p = 2$ becomes more complicated.  Details can be found in \cite[Proposition 9.3.4]{LevinThesis}.
\end{proof} 

Choose a crystalline object  $\widehat{\fP}_{\F} \in \GMod^{\phz, [a,b], \widehat{\Gamma}}_{\fS_{\F}}$. If $\fP_{\F}$ is the underlying $G$-Kisin module of $\widehat{\fP}_{\F}$, then we would like to study the forgetful functor
$$
\widehat{\Delta}:D^{\text{cris}, [a,b]}_{\widehat{\fP}_{\F}} \ra D^{[a,b]}_{\fP_{\F}}.
$$
More specifically, if $\mu$ and $a, b$ are as in the discussion before Definition \ref{muDbar} and $F = F_{[\mu]}$, we consider
$$
 \widehat{\Delta}^{\mu}: D^{\text{cris}, \mu}_{\widehat{\fP}_{\F}} := D^{\text{cris}, [a,b]}_{\widehat{\fP}_{\F}} \times_{ D^{[a,b]}_{\fP_{\F}}} D^{\mu}_{\fP_{\F}} \ra D^{\mu}_{\fP_{\F}}.
$$

We can now state our main theorem:
\begin{thm} \label{keythm} Assume that $p$ does not divide $\pi_1(G^{\der})$ where $G^{\der}$ is the derived group of $G$ and that $F = F_{[\mu]}$. If $\mu$ is a minuscule geometric cocharacter of $\Res_{(K \otimes_{\Qp} F)/F} G_F$ then 
$$
 \widehat{\Delta}^{\mu}:D^{\mathrm{cris}, \mu}_{\widehat{\fP}_{\F}} \ra D^{\mu}_{\fP_{\F}}
$$
is an equivalence of groupoids over $\cC_{\La}$. 
\end{thm}

\begin{rmk} This generalizes Theorem 9.3.13 in \cite{LevinThesis} where we worked with $G$-Kisin modules with height in $[0,1]$. See Remark \ref{assume} for more information. 
\end{rmk}

\begin{cor} \label{BTtype} Assume $F = F_{[\mu]}$ and that $\mu$ is minuscule.  Let $F'$ be finite extension of $F$ with ring of integers $\La'$.  There is an equivalence of categories between $G$-Kisin modules over $\fS_{\La'}$ with $p$-adic Hodge type $\mu$ and the subcategory of $\GRep_{\La'}(\Gamma_K)$ consisting of crystalline representations with $p$-adic Hodge type $\mu$. 
\end{cor}
Corollary \ref{BTtype} follows from the proof of Theorem \ref{keythm}.  It generalizes the equivalence between Kisin modules of Barsotti-Tate type and lattices in crystalline representations with Hodge-Tate weights in $\{-1, 0\}$ (\cite[Theorem 2.2.7]{Fcrystals}). Note that we do not require $p \nmid |\pi_1(G^{\der})|$ here.  For the relevant definitions, see  Definition \ref{defnHtype2} and the discussion before Theorem \ref{GLiu}. Before proving Theorem \ref{keythm} and Corollary \ref{BTtype}, we begin with some preliminaries on crystalline $(\phz, \widehat{\Gamma})$-modules with $G$-structure.  

\begin{defn} Define $G(u^{p^i})$ to be the kernel of the reduction map $G(W(\widetilde{\bE}^+)_A) \ra G(W(\widetilde{\bE}^+)_A /(\phz(\ft) u^{p^i}))$.
\end{defn}

\begin{prop} \label{Gcommconcrete} Choose $(\fP_A, \phi_{\fP_A}) \in \GMod_{\fS_A}^{\phz, \bh}$.  Fix a trivialization $\beta_A$ of $\fP_A$.  Let $C' \in G(\fS_A[1/(\phz(E(u))])$ be $\phi_{\fP_A}$ with respect to the trivialization $1 \otimes_{\phz} \beta_A$.  Then a crystalline $\widehat{\Gamma}$-structure on $\fP_A$ is the same as a continuous map
$$
B_{\bullet}:\widehat{\Gamma} \ra G( \widehat{R}_A)
$$
satisfying the following properties:
\begin{enumerate} [$(a)$]
\item  $C' \cdot \phz(B_{\gamma}) = B_{\gamma} \cdot \gamma(C') $ in $G(W(\widetilde{\bE})_A)$ for all $\gamma \in \widehat{\Gamma}$;
\item $B_{\gamma} = \Id$ for all $\gamma \in \widehat{\Gamma}_{\infty}$;
\item $B_{\gamma} \in G(u^{p})$ for all $\gamma \in \widehat{\Gamma}$;
\item $B_{\gamma \gamma'} = B_{\gamma} \cdot \gamma(B_{\gamma'})$ for all $\gamma, \gamma' \in \widehat{\Gamma}$.
\end{enumerate}
\end{prop}
\begin{proof}  Everything follows directly from Proposition \ref{commconcrete}.  The only remark to make is that because $u \in I_+(\widehat{R})$, $(u^p \phz(\ft)) \subset  I_+(\widehat{R})_A$.  Hence, the crystalline condition which is equivalent to condition (c) implies condition (5) from Definition \ref{hatmodule}.
\end{proof}

Before we begin the proof of Theorem \ref{keythm}, we have two important lemmas.

\begin{lemma} \label{conjugation}  Let $\fP_A \in  D^{\mu}_{\fP_{\F}}(A)$ and choose a trivialization $\beta_A$ of the bundle $\fP_A$.  If $C \in G(\fS_A[1/E(u)])$ is the Frobenius with respect to $\beta_A$, then for any $Y \in G(u^{p^i})$
$$
\phz(C)\phz(Y) \phz(C)^{-1} \in G(u^{p^{i+1}}),
$$
where $\phz(C) = C' \in G(W(\widetilde{\bE})_A)$ is the Frobenius with respect to $1 \otimes_{\phz} \beta_A$. 
\end{lemma}
\begin{proof} Let $\cO_{G}$ denote the coordinate ring of $G$ and let $I_e$ be the ideal defining the identity so that $\cO_G/I_e = \La$ and $I_e/I_e^2 \cong (\Lie (G))^{\vee}$.  Then $G(u^{p^i})$ is identified with 
$$
\{ Y \in \Hom_{\La} (\cO_G, W(\widetilde{\bE}^+)_A) \mid Y(I_e) \subset (\phz(\ft)u^{p^i}) \}.
$$

Conjugation by $C$ induces an automorphism of $G_{\fS_A[1/E(u)]}$.   Let $\Ad_{\cO_G}(C)^*:\cO_G \otimes_{\La} \fS_A[1/E(u)]  \ra \cO_G  \otimes_{\La} \fS_A[1/E(u)]$ be the corresponding map on coordinate rings.  The key observation is that
\begin{equation} \label{heightcontrol}
\Ad_{\cO_G}(C)^*(I_e \otimes 1) \subset \sum_{j \geq 1} I_e^j \otimes_{\La} E(u)^{-j} \fS_A.
\end{equation}
By successive approximation, one is reduced to studying the induced automorphism of $\oplus_{j \geq 0} (I_e^j/I_e^{j+1} \otimes_{\La} \fS_A[1/E(u)])$.  The $j$th graded piece is $\Sym^j(\Lie(G)^{\vee}) \otimes_{\La} \fS_A[1/E(u)]$ as a representation of $G(\fS_A[1/E(u)])$.  Since $\mu$ is minuscule, $\Lie(G) \otimes_{\La} \fS_A$ has height in $[-1, 1]$ and so $\Sym^j(\Lie(G)^{\vee} \otimes_{\La} \fS_A)$ has height in $[-j, j]$.  Thus, 
$$
\Ad_{\cO_G}(C)^*(\Sym^j(\Lie(G)^{\vee} \otimes_{\La} \fS_A) \subset E(u)^{-j} (\Sym^j(\Lie(G)^{\vee}) \otimes_{\La} \fS_A)
$$  
from which one deduces (\ref{heightcontrol}).   

Let $Y \in G(u^{p^i})$. Then $\phz(Y)(I_e) \subset \phz( \phz(\ft)u^{p^i}) \subset (\phz(E(u)) \phz(\ft) u^{p^{i+1}})$.    For any $x \in I_e$, 
$$
(\phz(C)\phz(Y) \phz(C)^{-1})(x) = (\phz(Y) \otimes 1)((1 \otimes \phz)(\Ad_{\cO_G}(C)^*(x)))
$$
which is a priori only in $W(\widetilde{\bE})_A$.  But since for any $b \in I_e^j$, $\phz(Y)(b)$ is divisible by $\phz(E(u))^j \phz(\ft)^j u^{jp^{i+1}}$, we have $\Ad(\phz(C))(\phz(Y))(x) \in  (\phz(\ft) u^{p^{i+1}})$ so $\phz(C)\phz(Y) \phz(C)^{-1}$ lies in $G(u^{p^{i+1}})$.
\end{proof}

By \cite[Corollary 1.3.15]{Fcrystals}, a $\Gamma_{\infty}$-representation coming from a finite height torsion-free Kisin module $\fM$ extends to a crystalline $\Gamma_K$-representation if and only if the canonical Frobenius equivariant connection on $\fM \otimes_{\fS} {\cO[1/\la]}$ has at most logarithmic poles. \cite[Proposition 2.2.2]{Fcrystals} states furthermore that if $\fM$ has height in $[0,1]$ then the condition of logarithmic poles is always satisfied.  The following lemma is a version of \cite[Proposition 2.2.2]{Fcrystals} for $G$-Kisin modules with minuscule type: 
\begin{lemma} \label{overF}  Let $F'/F$ be any finite extension containing $F_{[\mu]}$ and let $(\fP_{F'}, \phi_{F'})$ be any $G$-Kisin module over $F'$. Fix a trivialization of $\fP_{F'}$ and let $C \in G(\fS_{F'}[1/E(u)])$ be the Frobenius with respect to this trivialization.  If the $G$-filtration $\fD_{\fP_{F'}}$ over $K \otimes_{\Qp} F'$ defined before Lemma $\ref{Dx}$ has type $\mu$, then the right logarithmic derivative  $\frac{dC}{du} \cdot C^{-1} \in (\Lie G \otimes  \fS_{F'}[1/E(u)])$ has at most logarithmic poles along $E(u)$, i.e., lies in $E(u)^{-1}(\Lie G \otimes  \fS_{F'})$.
\end{lemma}
\begin{proof} 
Choose an embedding $\sigma:K_0 \ra F'$.  Without loss of generality, we assume that $\sigma(E(u))$ splits in $F'$ and write $\sigma(E(u)) = (u- \psi_1(\pi))(u - \psi_2(\pi)) \ldots (u- \psi_e(\pi))$ over embeddings $\psi_i:K \ra F'$ which extend $\sigma$.  Let $C_{\sigma}$ denote the $\sigma$-component of $C$ under the decomposition of $\fS_{F'}[1/E(u)]$ as a $W \otimes_{\Zp} F' \cong \prod_{K_0 \ra F'} F'$-algebra.   We can furthermore compute the ``pole" at $\psi_i(\pi)$ by working in the completion at $u - \psi_i(\pi)$ which is isomorphic to $F'[\![t]\!]$ with $t = u - \psi_i(\pi)$.  

Let $\mu_{\psi_i} \in X_*(G_F)$ be the $\psi_i$-component of $\mu$.  Fix a maximal torus $T$ of $G_{F'}$ such that $\mu_{\psi_i}$ factors through $T$. The Cartan decomposition for $G(F'(\!(t)\!))$ combined with the assumption that $\fD_{\fP_{F'}}$ has type $\mu$ implies that 
$$
C_{\sigma} = B_i \mu_{\psi_i}(t) D_i
$$
where $B_i$ and $D_i$ are in $G(F'[\![t]\!])$ (see discussion before Proposition \ref{Adpole} for definition of $\mu_{\psi_i}(t)$).  Finally, we compute that $\frac{dC_{\sigma}}{du} C_{\sigma}^{-1}$ equals
$$
\frac{dB_i}{dt} B_i^{-1} + \Ad(B_i)\left(\frac {d\mu_{\psi_i}(t)}{dt} \mu_{\psi_i}(t)^{-1}\right)  +  \Ad(B_i) \left(\Ad (\mu_{\psi_i}(t)) \left(\frac{dD_i}{dt}  D_i^{-1}\right) \right).     
$$
We have $\frac{dB_i}{dt} B_i^{-1} \in (\Lie G \otimes F'[\![t]\!])$. Using a faithful representation on which $T$ acts diagonally, we have $\frac{d\mu_{\psi_i}(t)}{dt} \mu_{\psi_i}(t)^{-1} \in \frac{1}{t} (\Lie G \otimes F'[\![t]\!])$.   Finally, since $\mu_{\psi_i}$ is minuscule, $\Ad (\mu_{\psi_i}(t)) (X) \in \frac{1}{t} (\Lie G \otimes F'[\![t]\!])$ for any $X \in \Lie G$ so in particular for $\frac{dD_i}{dt} D_i^{-1}$ by Proposition \ref{Adpole}. 
\end{proof} 
\begin{proof}[Proof of Theorem \ref{keythm}]
The faithfulness of $\widehat{\Delta}^{\mu}$ is clear.  For fullness, let $\widehat{\fP}_A, \widehat{\fP}_A' \in D^{\text{cris}, \mu}_{\widehat{\fP}_{\F}}(A)$ and let $\psi:\fP_A \cong \fP'_A$ be an isomorphism of underlying $G$-Kisin modules.  To show $\psi$ is equivariant for the $\widehat{\Gamma}$-actions, we can identify $\fP_A$ and $\fP_A'$ using $\psi$ and choose a trivialization of $\fP_A$.  Then, it suffices to show that $(\fP_A, \phi_{\fP_A})$ has at most one crystalline $\widehat{\Gamma}$-structure.  Let $B_{\tau}$ and $B'_{\tau}$ in $G(W(\widetilde{\bE}^+)_A)$ define the action of $\tau$ with respect to the chosen trivialization of $\phz^*(\fP_A)$ for the two $\widehat{\Gamma}$-structures.  By the crystalline property, $B_{\tau} (B'_{\tau})^{-1} \in G(u^{p})$.  By Proposition \ref{commconcrete} if Frobenius is given by $C'$ with respect to the trivialization, 
$$
B_{\tau} (B'_{\tau})^{-1} = C' \phz(B_{\tau} (B'_{\tau})^{-1}) (C')^{-1}.
$$
But then by Lemma \ref{conjugation},  $B_{\tau} (B'_{\tau})^{-1}  = I$ since it is in $G(u^{p^i})$ for all $i \geq 1$.  

We next attempt to construct a crystalline $\widehat{\Gamma}$-structure on any $\fP_A \in D^{\mu}_{\fP_{\F}}(A)$.  Along the way, we will have to impose certain closed conditions on $D^{\mu}_{\fP_{\F}}$ to make our construction work.  In the end, we will reduce to $A$ flat over $\Zp$ to show that these conditions are always satisfied. Fix a trivialization $\beta_A$ of $\fP_A$. We want elements $\{ B_{\gamma} \} \in G(\widehat{R}_A)$ for all $\gamma \in \widehat{\Gamma}$ satisfying the conditions from Proposition \ref{Gcommconcrete}. Choose an element $\gamma \in \widehat{\Gamma}$.  Let $C$ denote the Frobenius with respect to $\beta_A$ and let $C' = \phz(C)$ be the Frobenius with respect to $1 \otimes_{\phz} \beta_A$.

We use the topology on $G(W(\widetilde{\bE})_A)$ induced from the topology on $W(\widetilde{\bE})_A$ (see the discussion before Definition \ref{hatmodule}). Take $B_0 = I$. For all $i \geq 1$,  define
\begin{equation} \label{commXY}
B_i := C' \phz(B_{i-1}) \gamma(C')^{-1} \in G(W(\widetilde{\bE})_A).
\end{equation}
If $\fP_A$ admits a $\widehat{\Gamma}$-structure, then the $B_i$ converge to $B_{\gamma}$ in $G(\widehat{R}_A)$ or equivalently in $G(W(\widetilde{\bE})_A)$. 

\textbf{Base case:}  $B_1 = C' \gamma(C')^{-1} \in G(u^{p}).$  Let $V$ be a faithful $n$-dimensional representation of $G$ such that $\fP_A(V)$ has height in $[a,b]$. Set $r = b-a$. Consider $C$ as an element of $\GL_n(\fS_A[1/E(u)])$ such that
$$
C'' := E(u)^{-a} C \in \Mat_n(\fS_A) \text{ and } D'' := E(u)^b C^{-1} \in \Mat_n(\fS_A)
$$  
with $C'' D'' = E(u)^{r} I$. Working in $\Mat_n( W(\widetilde{\bE})_A)$, we compute that
$$
C' \gamma(C')^{-1} - I = \phz  \left ( \frac{1}{E(u)^{-a} \gamma(E(u))^b} (C'' \gamma(D'') - E(u)^{-a} \gamma(E(u))^b I) \right ).
$$
It would suffice then to show that $u \phz(\ft) E(u)^{r - 1}$ divides $C'' \gamma(D'') - E(u)^{-a} \gamma(E(u))^b I$ in $\Mat_n( W(\widetilde{\bE}^+)_A)$ as then $u \ft$ divides  $\frac{1}{E(u)^{-a} \gamma(E(u))^b} (C'' \gamma(D'') - E(u)^{-a} \gamma(E(u))^b I)$ using Lemma \ref{tdiv}.

Consider $P(u_1, u_2) = C''(u_1) D''(u_2)$ where we replace $u$ by $u_1$ in $C'' \in \Mat_n(\fS_A)$ and $u$ in $u_2$ for $D''$.  Let $P_{ij}(u_1, u_2) = \sum_{k \geq 0} c_k^{ij}(u_1) u_2^k$ be the $(i,j)$th entry where $c_k^{ij}(u_1)$ is  a power series in $u_1$ with coefficients in $W \otimes_{\Zp} A$.   We have that $P_{ij}(u,u) = \delta_{ij} E(u)^r$.    The $(i, j)$th entry of $C'' \gamma(D'')$ is 
$$
P_{ij}(u, [\underline{\eps}] u) = \sum_{k \geq 0} [\underline{\eps}] ^k c_k^{ij}(u) u^k
$$
where $\underline{\eps} = (\zeta_{p^i})_{i \geq 0}$ is the sequence of $p^n$-th roots of unity such that $\gamma(\pi^{1/p^n}) = \zeta_{p^n} \pi^{1/p^n}$.  Note that $ \phz(\ft)$ divides $[\underline{\eps}] - 1$ since $[\underline{\eps}] - 1 \in I^{[1]} W(\widetilde{\bE}^+)$ (see \cite[Proposition 5.1.3]{FonPP}) and $\phz(\ft)$ is a generator for this ideal.  Then, 
$$
P_{ij}(u, [\underline{\eps}] u) = \sum_{k \geq 0} ([\underline{\eps}]^k - 1) c_k^{ij}(u) u^k  + \delta_{ij} E(u)^r.
$$
 Since $u ([\underline{\eps}] - 1) E(u)^{r-1}$ divides $E(u)^r - E(u)^{-a} \gamma(E(u))^b$, it suffices to show that $u ([\underline{\eps}] - 1) E(u)^{r-1}$ divides  $\sum_k ([\underline{\eps}]^k - 1) c_k^{ij}(u) u^k$.  Using the Taylor expansion for $x^k - 1$ at $x = 1$, we have
$$
[\underline{\eps}]^k - 1 = \sum_{\ell=1}^{k} \binom{k}{\ell} ([\underline{\eps}] - 1)^{\ell}
$$
from which we deduce that
$$
 \sum_{k \geq 0} ([\underline{\eps}]^k - 1) c_k^{ij}(u) u^k = u ([\underline{\eps}] - 1) \left( \sum_{\ell \geq 1}  ([\underline{\eps}] - 1)^{\ell -1} u^{\ell - 1} \sum_{k \geq 0} \binom{k+\ell}{\ell} c_{k +\ell}^{ij}(u) u^{k} \right)
$$
Since $E(u)$ divides $[\underline{\eps}] - 1$, we are reducing to showing that $$E(u)^{r-\ell} \mid u^{\ell - 1} \sum_{k \geq 0} \binom{k+\ell}{\ell} c_{k +\ell}^{ij}(u) u^{k}$$ for $1 \leq \ell \leq r -1$ where the expression on the right is exactly  $\frac{u^{\ell - 1}}{\ell !} \left( \frac{d^{\ell} P_{ij}(u_1, u_2)}{du_2^{\ell}}|_{(u, u)} \right)$.   

Let $(\star_1)$ be the condition that $E(u)^{r-\ell}$ divides $\frac{d^{\ell} P_{ij}(u_1, u_2)}{du_2^{\ell}}|_{(u, u)}$ for all $(i, j)$ and $1 \leq \ell \leq r -1$.   This is a closed condition on $D^{\mu}_{\fP_{\F}}$.   
 
\textbf{Induction step:} Let $\fP_A \in D^{\mu}_{\fP_{\F}}(A)$ satisfying $(\star_1)$ with trivialization as above so that $B_1 = C' \gamma(C')^{-1} \in G(u^{p}).$  We have
$$
B_{i+1} B_i^{-1} = C \phz(B_i B_{i-1}^{-1}) C^{-1}.
$$
As $C = \phz(C')$, we can apply Lemma \ref{conjugation} to conclude that $B_{i+1} B_i^{-1} \in G(u^{p^{i+1}})$, i.e., $B_{i+1} B_i^{-1} \equiv I \mod \phz(\ft) u^{p^{i+1}} W(\widetilde{\bE}^+)_A$.   Since $W(\widetilde{\bE}^+)_A$ is separated and complete, $\varinjlim B_i = B_{\gamma} \in G(W(\widetilde{\bE}^+)_A)$ and $B_{\gamma}$ satisfies $B_{\gamma} \gamma(C) = C \phz(B_{\gamma})$.  It is easy to see that for any $\gamma, \gamma'$,  $B_{\gamma} \gamma'(B_{\gamma}) = B_{\gamma \gamma'}$ by continuity so we have a $\widehat{\Gamma}$-action.   If $\gamma \in \widehat{\Gamma}_{\infty}$, then $\gamma$ acts trivially on $\fS_A$ and so on $C$ as well so $B_{\gamma} = I$.

Let $(\star_2)$ denote the condition that $B_{\gamma} \in G(\widehat{R}_A)$ for all $\gamma \in \widehat{\Gamma}$.  We claim this is also a closed condition on $D^{\mu}_{\fP_{\F}}$.  Since $W(\widetilde{\bE}^+)/\widehat{R}$ is $\Zp$-flat, the sequence 
$$
0 \ra \widehat{R}_A \ra W(\widetilde{\bE}^+)_A \ra (W(\widetilde{\bE}^+)/\widehat{R}) \otimes_{\Zp} A \ra 0
$$
is exact for any $A$.  Any flat module over an Artinian ring is free so vanishing of an element $f \in (W(\widetilde{\bE}^+)/\widehat{R}) \otimes_{\Zp} A$ is a closed condition on $\Spec A$.  

We have shown that any element $\fP_A \in D^{\mu}_{\fP_{\F}}(A)$ which satisfies $(\star_1)$ and $(\star_2)$ admits a crystalline $\widehat{\Gamma}$-structure and so lies in $D^{\text{cris}, \mu}_{\widehat{\fP}_{\F}}(A)$.  It suffices then to show that the closed subgroupoid defined by the conditions $(\star_1)$ and $(\star_2)$ is all of $D^{\mu}_{\fP_{\F}}$.  Recall that $D^{\mu}_{\fP_{\F}}$ admits a formally smooth representable hull $D^{(N), \mu}_{\fP_{\F}} = \Spf R^{(N), \mu}_{\fP_{\F}}$ where $R^{(N), \mu}_{\fP_{\F}}$ is flat and reduced by Theorem \ref{locmodels} and Proposition \ref{musmoothfinite}. Since $R^{(N), \mu}_{\fP_{\F}}$ is flat and $R^{(N), \mu}_{\fP_{\F}}[1/p]$ is reduced and Jacobson, any closed subscheme of $\Spec R^{(N), \mu}_{\fP_{\F}}$ which contains $\Hom_{\La} (R^{(N), \mu}_{\fP_{\F}}, F')$ for all $F'/F$ finite is the whole space.  It suffices then to show that for any $F'/F$ finite and $\La'$ the ring of integers of $F'$ every object of  $D^{\mu}_{\fP_{\F}}(\La')$ satisfies $(\star_1)$ and $(\star_2)$.   

For $(\star_1)$, choose $\gamma \in \widehat{\Gamma}$.  Then, set $Q_{\ell}(u) :=  \left( \frac{d^{\ell} P_{ij}(u_1, u_2)}{du_2^{\ell}}|_{(u, u)} \right) \in \Mat_n(\fS_{\La'})$ (we ignore $\frac{u^{\ell - 1}}{\ell !}$ since we are in the torsion-free setting).  We can check that $E(u)^{r-\ell} \mid Q_{\ell}(u)$ working over $F' = \La'[1/p]$ or any finite extension thereof.  In particular, we can put ourselves in the situation of Lemma \ref{overF}.  We compute then that
\begin{equation*}
\begin{split}
Q_{\ell}(u) &=  (E(u)^{-a} C) \frac{d^{\ell}}{du^{\ell}} (E(u)^b C^{-1})\\
&= (E(u)^{-a} C) \sum^{\ell}_{m = 0} \binom{\ell}{m} \frac{d^{m} E(u)^b}{du^{m}}   \frac{d^{\ell - m} C^{-1}}{du^{\ell - m}}\\
&= \sum^{\ell}_{m = 0} \binom{\ell}{m} \left( E(u)^{-a} \frac{d^{m} E(u)^b}{du^{m}}\right)  \left(C  \frac{d^{\ell - m} C^{-1}}{du^{\ell - m}}\right).
\end{split}
\end{equation*}
Since $E(u)^{r - m}$ divides $E(u)^{-a} \frac{d^{m} E(u)^b}{du^{m}}$, it suffices to show that 
$$
Y_k := E(u)^k \left(C \frac{d^{k} C^{-1}}{du^{k}}\right) \in \Mat_n(\fS_{F'})
$$
for all $k \geq 0$ (applied with $k = \ell - m$).  The case $k = 0$ is trivial.  By Lemma \ref{overF},  $X_C := E(u) \frac{dC}{du} C^{-1} = - E(u) C \frac{d (C^{-1})}{du}$ is an element of $\Lie G \otimes \fS_{F'}$ considered as subset of $\Lie (\GL(V)) \otimes \fS_{F'}$ so in particular $Y_1 \in \Mat_n(\fS_{F'})$.  The product rule applied to $\frac{d}{du}(E(u)^k C \frac{d^{k-1} C^{-1}}{du})$ implies that
$$
Y_k = \frac{d}{du} (E(u) Y_{k-1}) - k \frac{d E(u)}{du} Y_{k-1} + Y_1 Y_{k -1}
$$ 
so by induction on $k$, $Y_k \in \Mat_n(\fS_{F'})$ for all $k \geq 0$.

For $(\star_2)$, recall that $\widehat{R} = R_{K_0} \cap W(\widetilde{\bE}^+)$ (see pg. 5 of \cite{Liu1}) so it suffices to show that $B_{\gamma} \in G(R_{K_0} \otimes_{\Zp} \La')$ or equivalently $B_{\gamma} \in \GL_n(R_{K_0} \otimes_{\Zp} \La')$ with respect to $V$.  Denote by $\fM_V$ the Kisin module $\fP_{\La'}(V)$ of rank $n$ . Since $\phz(E(u))$ is invertible in $S_{K_0} $, $C'$ lies in $\GL_n(S_{K_0} \otimes_{\Zp} \La')$ and defines a Frobenius on the Breuil module $\cM_V := S_{K_0} \otimes_{\fS, \phz} \fM_V $. Using a similar argument to above, one can construct the monodromy operator $N_{\cM_V}$ on $\cM_V$ inductively taking $N_0 = 0$ and setting 
\begin{equation} \label{Ncomm}
N_{i+1} := p C' \phz(N_i)( C')^{-1}  + u \frac{dC'}{du} (C')^{-1}.
\end{equation}
The sequence $\{ N_i \}$ converges to an element of $\Mat_n(u^p S_{K_0})$. For each $N_i$, let $\widetilde{N}_i$ be the induced derivation on $\cM_V$ over $-u \frac{d}{du}$ which on the chosen basis is given by $N_i$. Equation (\ref{Ncomm}) is equivalent to
\begin{equation} \label{Ncomm2}
\widetilde{N}_{i+1} \phi_{\cM_V} = p \phi_{\cM_V} \widetilde{N}_{i}.
\end{equation}
Let  $\underline{\eps}(\gamma) := \gamma([\underline{\pi}] )/[\underline{\pi}]$. Define a $\gamma$-semilinear map $\widetilde{B}_i$ on $R_{K_0} \otimes_{S_{K_0}} \cM_V$ by 
$$
\widetilde{B}_i(x) = \sum_{j \geq 0} \frac{(- \log \underline{\eps}(\gamma))^j}{j!}\otimes (\widetilde{N}_i)^j (x) 
$$
for all $x \in \cM_V$. Equation (\ref{Ncomm2}) implies that
$$
\widetilde{B}_{i + 1} \phi_{\cM_V} = \phi_{\cM_V} \widetilde{B}_i.
$$
By induction on $i$, one deduces that $\widetilde{B}_i$ is exactly the $\gamma$-semilinear morphism induced by the matrix $B_i$ defined in (\ref{commXY}).   

If $N_{\cM_V}$ is the limit of the $\widetilde{N}_i$ and $\widetilde{B}_{\gamma}$ is the $\gamma$-semilinear morphism induced by $B_{\gamma}$, then we have the following formula
$$
\widetilde{B}_{\gamma}(x) := \sum_{j \geq 0} \frac{(- \log \underline{\eps}(\gamma))^j}{j!}\otimes N_{\cM}^j (x)
$$
for all $x \in \cM_V$.   Working with respect to the chosen basis for $\cM_V$, we deduce that $B_{\gamma} \in \GL_n(R_{K_0} \otimes_{\Zp} \La')$ as desired.
\end{proof} 

\subsection{Applications to $G$-valued deformation rings} 

Let $\overline{\eta}:\Gamma_K \ra G(\F)$ be a continuous representation.  As before, $\mu$ is a minuscule geometric cocharacter of $\Res_{(K \otimes_{\Qp} F)/F} G_F$. Let $R^{\cris, \mu}_{\overline{\eta}}$ be the univeral $G$-valued framed crystalline deformation ring with $p$-adic Hodge type $\mu$ over $\La_{[\mu]}$.  Let $X^{\cris, \mu}_{\overline{\eta}}$ be the projective $R^{\cris, \mu}_{\overline{\eta}}$-scheme as in Corollary \ref{cristypes}.  The following theorem on the geometry of $X^{\cris, \mu}_{\overline{\eta}}$ has a number of important corollaries.  The proof uses the main results from \S 3.2 and \S 4.2. We can say more about the connected components when $K$ is unramified over $\Qp$ (see Theorem \ref{connectedunramified}).  

\begin{thm} \label{main} Assume $p \nmid \pi_1(G^{\der})$. Let $\mu$ be a minuscule geometric cocharacter of $\Res_{(K \otimes_{\Qp} F)/F} G_F$. Then $X^{\cris, \mu}_{\overline{\eta}}$ is normal and $X^{\cris, \mu}_{\overline{\eta}} \otimes_{\La_{[\mu]}} \F_{[\mu]}$ is reduced. 
\end{thm} 

\begin{cor} \label{connectedcomponents} Assume $p \nmid \pi_1(G^{\der})$. Let $X^{\cris, \mu}_{\overline{\eta}, 0}$ denote the fiber of $X^{\cris, \mu}_{\overline{\eta}}$ over the closed point of $\Spec R^{\cris, \mu}_{\overline{\eta}}$.  The connected components of  $\Spec R^{\cris, \mu}_{\overline{\eta}}[1/p]$ are in bijection with the connected components of $X^{\cris, \mu}_{\overline{\eta}, 0}$. 
\end{cor}
\begin{proof}
By Theorem \ref{stlocus}, $\Spec R^{\cris, \mu}_{\overline{\eta}}[1/p] = X^{\cris, \mu}_{\overline{\eta}}[1/p].$ Since $X^{\cris, \mu}_{\overline{\eta}} \otimes_{\La} \F$ is reduced (by Theorem \ref{main}), the bijection between $\pi_0(X^{\cris, \mu}_{\overline{\eta}}[1/p])$ and $\pi_0(X^{\cris, \mu}_{\overline{\eta},0})$ follows from the ``reduced fiber trick'' \cite[2.4.10]{MFFGS}.
\end{proof}

\begin{rmk} Both Theorem \ref{main} and Corollary \ref{connectedcomponents} hold for unframed $G$-valued crystalline deformation functors when they are representable by exactly the same arguments.
\end{rmk}

Before we begin the proof, we introduce a few auxiliary deformation groupoids.  The relationship between various deformation spaces is described in the diagram below.  Let $D^{\square}_{\overline{\eta}}$ be the deformation functor of $\overline{\eta}$, that is, $D^{\square}_{\overline{\eta}}(A)$ is the set of homomorphism $\eta:\Gamma_K \ra G(A)$ lifting $\overline{\eta}$.   Let $\fP_{\F}$ be the $G$-Kisin module associated to a $\F$-point $\overline{x}$ of $X^{\cris, \mu}_{\overline{\eta}}$ .
    
\begin{defn} \label{framedversions} Define $D^{[a,b]}_{\overline{x}}(A)$ to be the category of triples
$$
 \{\eta_A \in D^{\square}_{\overline{\eta}}(A) , \fP_A \in D^{[a,b]}_{\fP_{\F}}(A), \delta_A:T_{G,\fS_A}(\fP_A) \cong \eta_A|_{\Gamma_{\infty}} \}.
$$
Let $\widehat{\fP}_{\F}$ be a crystalline $\widehat{\Gamma}$-structure on $\fP_{\F}$ together with an isomorphism $\widehat{T}_{G, \F}(\widehat{\fP}_{\F}) \cong \overline{\eta}$. Define $D^{\text{cris}, \mu, \square}_{\widehat{\fP}_{\F}}(A)$ to be the category of triples
$$ 
\{\eta_A \in D^{\square}_{\overline{\eta}}(A) , \widehat{\fP}_A \in D^{\cris, \mu}_{\widehat{\fP}_{\F}}(A), \delta_A:\widehat{T}_{G,A}(\widehat{\fP}_A) \cong \eta_A \}.
$$
\end{defn}

\begin{prop} \label{hatfaithful} For any $\widehat{\fP}_{\F}$, the forgetful functor from $D^{\cris, \mu, \square}_{\widehat{\fP}_{\F}}$ to $D^{[a,b]}_{\overline{x}}$ is fully faithful.  
\end{prop} 
\begin{proof}
One reduces immediately to the case of $\GL_n$ and then we have the following more general fact: Choose any $\widehat{\fM}'_A, \widehat{\fM}_A \in \Mod^{\phz, \bh, \widehat{\Gamma}}_{\fS_A}$.  Let $f:\fM'_A \ra \fM_A$ be a map of underlying Kisin modules such that $T_{\fS_A}(f)$ is $\Gamma_K$-equivariant $($under the identification of $\widehat{T}_{\fS_A} \cong T_{\fS_A})$. Then, $f$ is a map of $(\phz, \widehat{\Gamma})$-modules.  This is proven in \cite[Corollary 4.3]{Ozeki} when height is in $[0,h]$ but can be easily extended to bounded height.  The key input is a weak form of Liu's comparison isomorphism (\cite[3.2.1]{Liu07} which can be found in \cite[Proposition 9.2.1]{LevinThesis}. 
\end{proof} 

The diagram below illustrates some of the relationships between the different deformation problems.  The diagonal maps on the left and the map labeled sm are formally smooth.  Maps labeled with $c \sim$ indicate that the complete stalk at a point of the target represents that deformation functor.   The horizontal equivalences are consequences of the Theorem \ref{keythm} and  the proof of Theorem \ref{main}  respectively. 
    
\begin{equation} \label{bigdiagram}
\xymatrix{
& \widetilde{D}^{(\infty),\mu}_{\fP_{\F}} \ar[dl]_{\pi^{\mu}} \ar[dr]^{\Psi^{\mu}} & & D^{\cris, \mu, \square}_{\widehat{\fP}_{\F}}  \ar[d]^{sm} \ar@{^{(}->}[dr] \ar[r]^{\sim} & D^{\cris, \mu}_{\overline{x}} \ar[r]^{c \sim}  \ar@{^{(}->}[d]  & X^{\cris, \mu}_{\overline{\eta}} \ar@{^{(}->}[d] \\
\overline{D}^{\mu}_{Q_{\F}}& & D^{\mu}_{\fP_{\F}} & D^{\cris, \mu}_{\widehat{\fP}_{\F}} \ar[l]_{\sim}  & D_{\overline{x}}^{[a,b]} \ar[r]^{c \sim}  & X^{[a,b]}_{\overline{\eta}}
}
\end{equation}

\begin{proof}[Proof of Theorem \ref{main}]  Let $\overline{x}$ be a point of the special fiber of $X^{\cris,  \mu}_{\overline{\eta}}$ defined over a finite field $\F'$.  Since $X^{\cris, \mu}_{\overline{\eta}}[1/p] = \Spec R^{\cris, \mu}_{\overline{\eta}}[1/p]$ is formally smooth over $F$ (\cite[Proposition 5.1.5]{Balaji}), it suffices to show that the completed stalk $\widehat{\cO}^{\mu}_{\overline{x}}$ at $\overline{x}$ is normal and that $\widehat{\cO}^{\mu}_{\overline{x}} \otimes_{\La_{[\mu]}} \F_{[\mu]}$ is reduced.  To accomplish this, we compare $\widehat{\cO}^{\mu}_{\overline{x}}$ with $\overline{D}^{\mu}_{Q_{\F'}}$ from \S 3.3 and then use as input the corresponding results for the local model $M(\mu)$.

These properties can be checked after an \'etale extension of $\La_{[\mu]}$.  $R^{\cris, \mu}_{\overline{\eta}}$ commutes with changing coefficients using the abstract criterion in \cite[Proposition 1.4.3.6]{CMbook} as does the formation of $X^{\cris, \mu}_{\overline{\eta}}$ by Proposition \ref{fhfunctoriality}.  We can assume then, without loss of generality, that $\La = \La_{[\mu]}$ and $\F' = \F$.  Let $\fP_{\F}$ be the $G$-Kisin module defined by $\overline{x}$. Since $\mu$ is minuscule, $X^{\cris, \mu}_{\overline{\eta}} = X^{\cris, \leq \mu}_{\overline{\eta}}$ (see Corollary \ref{closedschubert}).  

Since $\widehat{\cO}^{\mu}_{\overline{x}}$ is non-empty and $\La$-flat (assuming that $R^{\cris, \mu}_{\overline{\eta}}$ is non-empty), it has an $F'$-point for some finite extension $F'/F$.  Any such point gives rise to a crystalline lift $\rho$ of $\overline{x}$ to $\cO_{F'}$ such that the unique Kisin lattice in $\underline{M}_{G, \cO_{F'}}(\rho)$ reduces to $\fP_{\F} \otimes_{\F} \F'$.  Replace $\F'$ by $\F$.  Then by Proposition \ref{criscris}, the corresponding $G(\cO_{F'})$-valued representation is isomorphic to $\widehat{T}_{G, \cO_{F'}}(\widehat{\fP}_{\cO_{F'}})$ for some crystalline $(\phz, \widehat{\Gamma})$-module with $G$-structure.  Reducing modulo the maximal ideal, we obtain a crystalline $\widehat{\Gamma}$-structure $\widehat{\fP}_{\F}$ on $\fP_{\F}$.   By Proposition \ref{hatfaithful}, this is the unique such structure.  

Recall the deformation problem $D^{\cris, \mu}_{\overline{x}}$ from Corollary \ref{cristypes} and $D^{[a,b]}_{\overline{x}}$ from Definition \ref{framedversions}.  The natural map 
$$
D^{\cris, \mu}_{\overline{x}} \ra D^{[a,b]}_{\overline{x}}
$$
is a closed immersion (by Theorem \ref{stlocus}). By Corollary \ref{cristypes}, $\Spf \cO^{\mu}_{\overline{x}}$ is closed in $D^{\cris, \mu}_{\overline{x}}$. 

Fix the isomorphism $\beta_{\F}:\widehat{T}_{G, \F}(\widehat{\fP}_{\F}) \cong \overline{\eta}$. Consider the groupoid $D^{\text{cris}, \mu, \square}_{\widehat{\fP}_{\F}}$ from Definition \ref{framedversions}. There is a natural morphism then from $D^{\text{cris}, \mu, \square}_{\widehat{\fP}_{\F}}$ to $D^{[a,b]}_{\overline{x}}$ given by forgetting the $\widehat{\Gamma}$-structure.  By Proposition \ref{hatfaithful}, this morphism is fully faithful, hence a closed immersion by considering tangent spaces.

We claim that $D^{\text{cris}, \mu, \square}_{\widehat{\fP}_{\F}} =  \Spf \cO^{\mu}_{\overline{x}}$ as closed subfunctors of $D^{[a,b]}_{\overline{x}}$.  Since they are both representable, we look at their $F'$-points for any finite extension $F'$ of $F$.  By Theorem \ref{GLiu} and Corollary \ref{cristypes},
$$
D^{\text{cris}, \mu, \square}_{\widehat{\fP}_{\F}}(F') =  D^{\cris, \mu}_{\overline{x}}(F') = \Spf \cO^{\mu}_{\overline{x}}(F').
$$
Since $\cO^{\mu}_{\overline{x}}$ is $\La$-flat and $\cO^{\mu}_{\overline{x}}[1/p]$ is formally smooth over $F$, we deduce that  $\Spf \cO^{\mu}_{\overline{x}} \subset D^{\text{cris}, \mu, \square}_{\widehat{\fP}_{\F}}$.  

Finally, $D^{\text{cris}, \mu, \square}_{\widehat{\fP}_{\F}}$ is formally smooth over $D^{\mu}_{\fP_{\F}}$ by Theorem \ref{keythm}.  By  \ref{musmoothmod}, there exists a diagram
$$
\xymatrix{
& \Spf S^{\mu} \ar[dl]  \ar[dr] & \\
D^{\text{cris}, \mu, \square}_{\widehat{\fP}_{\F}}& & \overline{D}^{\mu}_{Q_{\F}}, \\
}
$$
where $S^{\mu} \in \widehat{\cC}_{\La}$ and both morphisms are formally smooth ($Q_{\F}$ is as in \S 3.2).  The functor $ \overline{D}^{\mu}_{Q_{\F}}$ is represented by a completed stalk $R^{\mu}_{Q_{\F}}$ on $M(\mu)$.  In particular,  $R^{\mu}_{Q_{\F}}$ is $\La$-flat so the same is true of $D^{\text{cris}, \mu, \square}_{\widehat{\fP}_{\F}}$. Thus, $D^{\text{cris}, \mu, \square}_{\widehat{\fP_{\F}}} =  \Spf \cO^{\mu}_{\overline{x}}$. By Theorem \ref{locmodels}, $R^{\mu}_{Q_{\F}}$ is normal, Cohen-Macaulay and $R^{\mu}_{Q_{\F}} \otimes_{\La} \F$ is reduced so the same is true for $\widehat{\cO}^{\mu}_x$.
\end{proof}

\begin{thm} \label{connectedunramified} Assume $K/\Qp$ is unramified, $p > 3$, and $p \nmid \pi_1(G^{\mathrm{ad}})$.  Then the universal crystalline deformation ring $R^{\cris, \mu}_{\overline{\eta}}$ is formally smooth over $\La_{[\mu]}$. 
\end{thm} 
\begin{proof} First, replace $\La$ by $\La_{[\mu]}$.  Without loss of generality, we can assume that $F$ contains all embeddings of $K$ since this can be arranged by a finite \'etale base change.  When $K/\Qp$ is unramified, $\Gr_G^{E(u), W}$ is a product of $[K:\Qp]$ copies of the affine Grassmanian $\Gr_G$ (see \cite[Proposition 10.1.11]{LevinThesis}). If $\mu = (\mu_{\psi})_{\psi:K \ra F}$, then $M(\mu)_F = \prod_{\psi} S(\mu_{\psi})$ where $S(\mu_{\psi})$ are affine Schubert varieties of $\Gr_{G_F}$. Under the assumption that $p \nmid \pi_1(G^{\mathrm{der}})$, there is a flat closed $\La$-subscheme of $\Gr_G$ which abusing notation we denote by $S(\mu_{\psi})$, whose fibers are the affine Schubert varieties for $\mu_{\psi}$ (see Theorem 8.4 in \cite{TwistedAG}, especially the discussions in \S 8.e.3 and 8.e.4). Thus,
$$
M(\mu) = \prod_{\psi:K \ra F} S(\mu_{\psi}).
$$
Since $\mu_{\psi}$ is minuscule, $S(\mu_{\psi})$ is isomorphic to a flag variety for $G$ hence $M(\mu)$ is smooth (see Proposition \ref{closedschubert}). The proof of Theorem \ref{main} shows that the local structure of $X^{\cris, \mu}_{\overline{\eta}}$ is smoothly equivalent to the local structure of $M(\mu)$. Thus, $X^{\cris, \mu}_{\overline{\eta}}$ is formally smooth over $\La$.

Finally, we have to show that
$$
\Theta:X^{\cris, \mu}_{\overline{\eta}} \ra \Spec R^{\cris, \mu}_{\overline{\eta}}
$$
is an isomorphism. Since $\Theta[1/p]$ is an isomorphism and $R^{\cris, \mu}_{\overline{\eta}}$ is $\La$-flat, it suffices to show that $\Theta$ is a closed immersion. Let $m_R$ be the maximal ideal of $R^{\cris, \mu}_{\overline{\eta}}$.  Consider the reductions
$$
\Theta_n:X^{\cris, \mu}_{\overline{\eta}, n} \ra \Spec R^{\cris, \mu}_{\overline{\eta}}/m_R^n.
$$
We appeal to an analogue of Raynaud's uniqueness result for finite flat models (\cite[Theorem 3.3.3]{Raynaud}). For any Artin local $\Zp$-algebra $A$ and any finite $A$-algebra $B$, let $\fP_B$ and $\fP_B'$ be two distinct points in the fiber of $\Theta_n$ over $x:R^{\cris, \mu}_{\overline{\eta}} \ra A$, i.e., $G$-Kisin lattices in $P_{x} \otimes_A B$.  Let $V^{\ad}$ denote the adjoint representation of $G$. Under the assumption that $p >3$,  \cite[Theorem 2.4.2]{Liu07} (which generalizes Raynaud's result) implies that $\fP_A(V^{\ad}) = \fP_A'(V^{\ad})$ as Kisin lattices in $(P_{x} \otimes_A B)(V^{\ad})$ using that $\mu$ is minuscule. 

Since $B$ is Artinian, without loss of generality we can assume it is local ring. Choose a trivialization of $\fP_B$. There exists $g \in G(\cO_{\cE, B})$ such that $\fP_{B}' = g.\fP_B$ (working inside the affine Grassmanian as in Theorem \ref{GKisinproj}).  The results above implies that $\Ad(g) \in G^{\ad}(\fS_A)$.  By assumption,  $Z := \ker(G \ra G^{\ad})$ is \'etale so after possibly extending the residue field $\F$ we can lift $\Ad(g)$ to an element $\widetilde{g} \in G(\fS_A)$  so that $g = \widetilde{g} z$ where $z \in Z(\cO_{\cE, A})$.  We want to show that $z \in Z(\fS_A)$. We can write $Z$ as a product $Z_{\tors} \times (\Gm)^s$ for some $s \geq 0$.  Since $Z_{\tors}$ has order prime to $p$ by assumption, $Z_{\tors}(\cO_{\cE, A}) = Z_{\tors}(\fS_A)$ so we can assume $$z \in (\Gm(\cO_{\cE, A}))^s = ((A \otimes_{\Zp} W)((u))^{\times})^s. $$  

For any embedding $\psi:W \ra \cO_F$, we associate to $z$ the $s$-tuple $\la_{\psi}$ of integers of the degrees of the leading terms of each component base changed by $\psi$. To show  that $\la_{\psi} = 0$ we can work over $A/m_A = \F$. We think of $\la_{\psi}$ as a cocharacter of $Z$.   Consider the quotient  of $G$ by its derived group $Z' := G/G^{\der}$.  The map $X_*(Z) \ra X_*(Z')$ is injective.   Any character $\chi$ of $Z'$ defines a one-dimensional representation $L_{\chi}$ of $G$ so in particular, we can consider $\fP_B(L_{\chi})$ and $\fP_B'(L_{\chi})$ as Kisin lattices in $P_x(L_{\chi})$. Writing $\fS_{\F} \cong \oplus_{\psi:W \ra \cO_F} \F[[u_{\psi}]]$, a Kisin lattice of $P_x(L_{\chi})$ has type $(h_{\psi})$ exactly when $\phi_{P_x}(e) = (a_{\psi} u^{h_{\psi}})e$ for a basis element $e$ and $a_{\psi} \in \F$.  Since both $\fP_B$ and $\fP_{B}'$ have type $\mu$, $\fP_B(L_{\chi})$ and $\fP_B'(L_{\chi})$ both have type $h_{\psi} := \langle \chi, \mu_{\psi} \rangle$.  However, a direct computation shows that $\fP_B'(L_{\chi})$ has type $h_{\psi} + \langle \chi, p \la_{\psi'} - \la_{\psi} \rangle$ where $\psi' = \phz \circ \psi$.  Thus, $\la_{\psi} = p \la_{\psi'}$.  We deduce that $p^{[K:\Qp]} \la_{\psi} = \la_{\psi}$ and so $\la_{\psi} = 0$.   

We are reduced to the following general situation: $X \ra \Spec A$ is proper morphism which is injective on $B$-points for all $A$-finite algebras $B$ where $A$ is a local Artinian ring.  By consideration of the one geometric fiber, $X \ra \Spec A$ is quasi-finite, hence finite.  Thus, $X = \Spec A'$. By Nakayama, it suffices to show $A/m_A \ra A'/(m_A) A'$ is surjective so we can assume $A = k$ is a field.  Surjectivity follows from considering the two morphisms $A' \rightrightarrows A' \otimes_{k} A'$ which agree by injectivity of $X \ra \Spec A$ on $A$-finite points.         
\end{proof}

\end{document}